\documentclass{amsart}
\usepackage[all]{xy}
\usepackage{graphicx}
\usepackage{amsmath}
\usepackage{amsfonts}
\usepackage{amssymb}

\newcommand{\bphi}{\bar{\phi}}
\newcommand{\onabla}{\overline{\nabla}}
\newcommand{\wphi}{\widetilde{\phi}}
\newcommand{\wnabla}{\widetilde{\nabla}}
\newcommand{\w}[1]{\widetilde{#1}}

\newtheorem{theorem}{Theorem}[section]

\newtheorem{corollary}[theorem]{Corollary}

\theoremstyle{definition}

\theoremstyle{remark} \theoremstyle{remark}
\newtheorem{remark}[theorem]{Remark}

\numberwithin{equation}{section}

\begin{document}

\title[Sasaki-Einstein and paraSasaki-Einstein metrics from $(\kappa,\mu)$-structures]{Sasaki-Einstein and paraSasaki-Einstein metrics\\  from $(\kappa,\mu)$-structures}

\author[B. Cappelletti-Montano]{Beniamino Cappelletti-Montano}
 \address{Dipartimento di Matematica e Informatica, Universt\`a degli Studi di
 Cagliari, Via Ospedale 72, 09124 Cagliari}
 \email{b.cappellettimontano@gmail.com}

\author[A. Carriazo]{Alfonso Carriazo}
 \address{Departamento de Geometr\'{i}a y Topolog\'{i}a, Universidad de Sevilla,
 Aptdo. de Correos 1160, 41080 Sevilla, SPAIN}
 \email{carriazo@us.es}

\author[V. Mart\'{i}n-Molina]{Ver\'onica  Mart\'{i}n-Molina}
 \address{Departamento de Geometr\'{i}a y Topolog\'{i}a, Universidad de Sevilla,
 Aptdo. de Correos 1160, 41080 Sevilla, SPAIN}
 \email{veronicamartin@us.es}

\begin{abstract}
We prove that every  contact metric $(\kappa,\mu)$-space admits a canonical $\eta$-Einstein Sasakian or $\eta$-Einstein paraSasakian metric. An explicit
expression for the curvature tensor fields of those metrics is given and we find the values of $\kappa$ and $\mu$ for which such metrics are
Sasaki-Einstein and paraSasaki-Einstein. Conversely, we prove that, under some natural assumptions, a K-contact or K-paracontact manifold foliated by two
mutually orthogonal, totally geodesic Legendre foliations admits a contact metric $(\kappa,\mu)$-structure. Furthermore, we apply the above results to the
geometry of tangent sphere bundles and we discuss some
 geometric properties of $(\kappa,\mu)$-spaces
related to the existence of Eistein-Weyl  and Lorentzian Sasaki-Einstein structures.
\end{abstract}

\subjclass[2010]{53C15, 53C12, 53C25, 53B30, 57R30}

\keywords{contact metric manifold,  Sasakian, paracontact, paraSasakian, nullity distribution, $(\kappa,\mu)$-spaces, Einstein, $\eta$-Einstein, Legendre
foliation, Weyl structure, Lorentzian Sasakian, tangent sphere bundle}

\thanks{This paper was started during a research visit of B.C.M. at the Department of Geometry and Topology of the University of Sevilla under a grant of the
Institute of Mathematics of the University of Sevilla (IMUS). B.C.-M. acknowledges the support of IMUS and expresses his gratitude for the warm hospitality
and good working conditions received at the Department of Geometry and Topology. The research of A.C. and V.M.-M. is partially supported by the PAI group
FQM-327 (Junta de Andaluc\'ia, Spain) and by the MTM2011-22621 grant of the MEC, Spain.}

\maketitle

\section{Introduction}

It is well known that the tangent sphere bundle $T_{1}N$ of a flat Riemannian manifold $N$ carries a contact Riemannian structure such that $R(X,Y)\xi=0$
for any vector fields $X$, $Y$ on $T_{1}N$, where the Reeb vector field $\xi$ is given by twice the geodesic flow. The class of contact metric manifolds
satisfying the above condition, which were at first studied by Blair in \cite{blair76}, is not preserved by  $\mathcal D$-homothetic transformations. In
fact, if one deforms $\mathcal D$-homothetically the structure, one falls in the larger class of ``contact metric  $(\kappa,\mu)$-spaces'', i.e. contact
metric manifolds $(M,\varphi,\xi,\eta,g)$  satisfying
\begin{equation}\label{definizione}
R(X,Y)\xi=\kappa\left(\eta\left(Y\right)X-\eta\left(X\right)Y\right)+\mu\left(\eta\left(Y\right)hX-\eta\left(X\right)hY\right),
\end{equation}
for some constants $\kappa$ and $\mu$, where $2h$ denotes the Lie derivative of the structure tensor $\varphi$ in the direction of the Reeb vector field
(see $\S$ \ref{first} for more details). This new class of Riemannian manifolds was introduced in \cite{blair95} as a natural generalization of both the
contact metric manifolds satisfying $R(X,Y)\xi=0$ and the Sasakian condition $R(X,Y)\xi=\eta\left(Y\right)X-\eta\left(X\right)Y$. Despite  the technical
appearance of the definition, nowadays contact $(\kappa,\mu)$-spaces are considered an important topic in contact Riemannian geometry  because there are
good reasons for studying them. The first is that, while the values $\kappa$ and $\mu$ vary, one proves that the condition \eqref{definizione} remains
unchanged  under $\mathcal D$-homothetic deformations. Next, in the non-Sasakian case (that is for $\kappa\neq 1$), the condition \eqref{definizione}
determines the curvature tensor field completely. Furthermore, $(\kappa,\mu)$-spaces provide non-trivial examples of some remarkable classes of contact
Riemannian manifolds, like
CR-integrable contact metric manifolds (\cite{tanno89}), H-contact manifolds (\cite{perrone}) and harmonic contact metric manifolds (%
\cite{vergara1}). Finally, there are non-trivial examples of such
Riemannian manifolds, the most important being the tangent sphere
bundle of any Riemannian manifold of constant sectional curvature
with its standard contact metric structure.

In this paper we study the relations between the theory of $(\kappa,\mu)$-spaces and  two other important topics of contact geometry: Sasakian
and paraSasakian manifolds. In fact, given a  non-Sasakian $(\kappa,\mu)$-space $(M,\varphi,\xi,\eta,g)$, we describe a method for constructing a Sasakian
or paraSasakian metric on $M$ compatible with the same contact form $\eta$. The type of metric (Sasakian or paraSasakian) depends on the value of a
well-known invariant introduced by Boeckx in \cite{boeckx2000} for classifying $(\kappa,\mu)$-spaces, defined as
\begin{equation}\label{boeckxnumber}
I_{M}=\frac{1-\frac{\mu}{2}}{\sqrt{1-\kappa}}.
\end{equation}
More precisely,  we are able to define a Sasakian or paraSasakian metric if $|I_M|>1$ or $|I_M|<1$, respectively. Moreover, by using the  aforementioned
property that the $(\kappa,\mu)$-nullity condition \eqref{definizione} determines the curvature completely, we  find an explicit expression for the
curvature tensor field of the above Sasakian and paraSasakian metrics. We obtain from it our main result, that such metrics are always $\eta$-Einstein and
that for some values of $\kappa$ and $\mu$ they are Sasaki-Einstein and paraSasaki-Einstein, though the starting $(\kappa,\mu)$-structure can  never be
Einstein in dimension greater than $3$ (\cite[p. 131]{blairbook}). Furthermore, we prove that in dimension greater than or equal to $5$, every
$(\kappa,\mu)$-space such that $I_M>1$ also carries an Einstein-Weyl structure.

We then discuss some consequences of  such results on the geometry of tangent sphere bundles $T_1 N$, which will accept $\eta$-Einstein Sasakian and paraSasakian metrics depending on the sign of $c$, the constant sectional curvature of the space form $N$. Moreover, these structures will be Sasaki-Einstein and paraSasaki-Einstein for certain values of $c$ (which will depend  only on $n$). Thus we extend the result of Tanno that $T_{1}S^{3}\simeq S^{2}\times S^{3}$ carries a Sasaki-Einstein metric and give (to the knowledge of the authors) the first non-trivial examples of $\eta$-Einstein (eventually Einstein) paraSasakian manifolds. In fact, while there has been an increasing interest in last years in paraSasakian geometry (see
\cite{alek09}, \cite{ivanov}, \cite{zamkovoy}), so far the only known examples of ($\eta$-)Einstein paracontact manifolds seem to be the hyperboloid
$\mathbb{H}^{2n+1}_{n+1}(1)$ of constant curvature $-1$ (\cite{ivanov}) and $\mathbb{R}^{3}_{1}$ with the flat metric (\cite{zamkovoy-arxiv2}), together
with the Boothby-Wang fibrations with base a paraK\"{a}hler-Einstein manifold.

Finally, in the last part of the paper we will give a geometric
interpretation to the above canonical Sasakian and paraSasakian
metrics. It is well known that any non-Sasakian $(\kappa,\mu)$-space
is foliated by two Legendre foliations, defined by the
eigendistributions of the operator $h$, and that such a foliated
structure plays an important role in the theory of
$(\kappa,\mu)$-spaces (cf. \cite{mino3}, \cite{mino6}). We show that
the geometry of these Legendre foliations, encoded by some
invariants like the Pang invariant (\cite{pang}) and the Libermann
map (\cite{libermann}), is fully described by the above Sasakian and
paraSasakian metrics. In this way we are able to find a sufficient
condition for a K-contact (respectively, K-paracontact) manifold
$M$, foliated by two mutually orthogonal, totally geodesic Legendre
foliations, to admit a contact metric $(\kappa,\mu)$-structure,
compatible with the same underlying contact form, such that
$|I_{M}|>1$ (respectively, $|I_{M}|<1$).

\section{Preliminaries}\label{first}

\subsection{Contact metric $(\kappa,\mu)$-spaces}
An \emph{almost contact structure} on a $(2n+1)$-dimensional smooth manifold $M$ is a triplet $(\varphi,\xi,\eta)$, where $\varphi$ is a tensor field of
type $(1,1)$, $\eta$ a $1$-form and $\xi$ a vector field on $M$ satisfying the following conditions
\begin{equation}\label{almostcontact}
\varphi^{2}=-I+\eta\otimes\xi, \ \ \eta(\xi)=1,
\end{equation}
where $I$ is the identity mapping. From \eqref{almostcontact} it follows that $\varphi\xi=0$, $\eta\circ\varphi=0$ and the $(1,1)$-tensor field $\varphi$
has constant rank $2n$ (\cite{blairbook}). Given an almost contact manifold $(M,\varphi,\xi,\eta)$ one can define an almost complex structure $J$ on the
product $M\times\mathbb{R}$ by setting $J\left(X,f\frac{d}{dt}\right)=\left(\varphi X-f\xi,\eta\left(X\right)\frac{d}{dt}\right)$ for any
$X\in\Gamma\left(TM\right)$ and $f\in C^{\infty}\left(M\times\mathbb{R}\right)$. Then the almost contact manifold is said to be \emph{normal} if the almost
complex structure $J$ is integrable. The condition for normality is given by the vanishing of the tensor field
$N_{\varphi}:=[\varphi,\varphi]+2d\eta\otimes\xi$. Any almost contact manifold $\left(M,\varphi,\xi,\eta\right)$ admits a \emph{compatible metric}, i.e. a
Riemannian metric $g$ satisfying
\begin{equation}\label{metric1}
g\left(\varphi X,\varphi Y\right)=g\left(X,Y\right)-\eta\left(X\right)\eta\left(Y\right)
\end{equation}
for all $X,Y\in\Gamma\left(TM\right)$. The manifold $M$ is  said to be an \emph{almost contact metric manifold} with structure
$\left(\varphi,\xi,\eta,g\right)$. From \eqref{metric1} it follows immediately that $\eta=g(\cdot,\xi)$ and $g(\cdot,\varphi\cdot)=-g(\varphi\cdot,\cdot)$.
Then one defines the $2$-form $\Phi$ on $M$ by $\Phi\left(X,Y\right)=g\left(X,\varphi Y\right)$, called the \emph{fundamental $2$-form} of the almost
contact metric manifold. If $\Phi=d\eta$ then $\eta$ becomes a contact form, with $\xi$ and ${\mathcal D}:=\ker(\eta)$ its corresponding \emph{Reeb vector
field} and \emph{contact distribution}, respectively. Then $(M,\varphi,\xi,\eta,g)$ is called a \emph{contact metric manifold}. In a contact metric
manifold one has
\begin{equation}\label{acca2}
\nabla\xi=-\varphi -\varphi h
\end{equation}
where $\nabla$ is the Levi-Civita connection of $(M,g)$ and $h$
denotes the $(1,1)$-tensor field defined by $h:=\frac{1}{2}{\mathcal
L}_{\xi}\varphi$. The tensor field $h$ is symmetric with respect to
$g$ and vanishes identically if and only if the Reeb vector field $\xi$ is
Killing. In this last case the contact metric manifold is said to be
\emph{K-contact}. A  normal contact metric manifold is called a
\emph{Sasakian manifold}. Any Sasakian manifold is K-contact and the
converse holds  in dimension $3$.

If the Ricci tensor of a contact metric manifold has the following
form
\begin{equation}\label{etaeinstein1}
\textrm{Ric} = a g + b \eta\otimes\eta
\end{equation}
for some functions $a$ and $b$, we say that $M$ is $\eta$-Einstein. It is known that if $M$ is Sasakian and $\dim(M)\geq 5$, then $a$ and $b$ are
necessarily constants. This notion appears to be a good generalization of the concept of Einstein metrics in the context of contact Riemannian geometry.
Many interesting geometric properties of $\eta$-Einstein metrics are presented in the recent paper \cite{galicki06}. Given a positive constant $c>0$, a
\emph{${\mathcal D}_{c}$-homothetic deformation} on a contact metric manifold $(M,\varphi,\xi,\eta,g)$ is the change of the structure tensors of the
following type
\begin{equation}\label{deformation}
\varphi':=\varphi, \ \ \ \xi':=\frac{1}{c}\xi, \ \ \ \eta':=c \eta,
\ \ \ g':=c g + c(c-1)\eta\otimes\eta.
\end{equation}
Then $(\varphi',\xi',\eta',g')$ is again a contact metric structure on $M$. A recent generalization of Sasakian manifolds is the notion of \emph{contact
metric $(\kappa,\mu)$-space} (\cite{blair95}). Let $(M,\varphi,\xi,\eta,g)$ be a contact metric manifold. If  the curvature tensor field of the Levi-Civita
connection satisfies \eqref{definizione} for some $\kappa,\mu\in\mathbb{R}$, we say that $(M,\varphi,\xi,\eta,g)$  is a \emph{contact metric
$(\kappa,\mu)$-space} (or that $\xi$ belongs to the $(\kappa,\mu)$-nullity distribution). This definition was introduced and deeply studied by Blair,
Koufogiorgos and Papantoniou in \cite{blair95}. Among other things, the authors proved the following result.

\begin{theorem}[\cite{blair95}]
Let $\left(M,\varphi,\xi,\eta,g\right)$ be a contact metric $(\kappa,\mu)$-space. Then
\begin{equation}\label{hformula1}
h^2 = -(1-\kappa)\varphi^2,
\end{equation}
so that necessarily $\kappa\leq 1$. Moreover, if $\kappa=1$ then $h=0$ and $\left(M,\varphi,\xi,\eta,g\right)$ is
 Sasakian. If $\kappa<1$, the contact metric structure is not
Sasakian and $M$ admits three mutually orthogonal integrable distributions ${\mathcal D}_{h}(0)=\mathbb{R}\xi$, ${\mathcal D}_{h}(\lambda)$ and ${\mathcal
D}_{h}(-\lambda)$ given by the eigenspaces of $h$ corresponding to the eigenvalues $0$, $\lambda$ and $-\lambda$, where $\lambda=\sqrt{1-\kappa}$.
\end{theorem}

The same authors also proved  the following formulas for the covariant derivatives of the tensor fields $\varphi$, $h$ and $\varphi h$ (\cite{blair95}):
\begin{align}
({\nabla}_X\varphi)Y=&g\left(X+{h}X,Y\right)\xi-\eta\left(Y\right)\left(X+{h}X\right),\label{integrabile1}\\
(\nabla_{X}h)Y=&((1-\kappa)g(X,\varphi Y)-g(X,\varphi hY))\xi+ \eta(Y)h(\varphi X+\varphi h X)-\mu\eta(X)\varphi h Y, \label{covariant1}\\
(\nabla_{X}\varphi h)Y=&(g(X,hY)-(1-\kappa)g(X,\varphi^{2}Y))\xi + \eta(Y)(hX-(1-\kappa)\varphi^{2}X) + \mu\eta(X)hY. \label{covariant2}
\end{align}
Notice that while $\mathcal D$-homothetic deformations preserve the state of being contact metric, K-contact, Sasakian or $\eta$-Einstein, they destroy
conditions like $R(X,Y)\xi=0$ or $R(X,Y)\xi=k(\eta(Y)X-\eta(X)Y)$. However, they preserve the class of contact metric $(\kappa,\mu)$-structures. Indeed, if
$(\varphi,\xi,\eta,g)$ is a $(\kappa,\mu)$-structure then the deformed structure $(\varphi',\xi',\eta',g')$ is a $(\kappa',\mu')$-structure with
\begin{equation}\label{deformation1}
\kappa'=\frac{\kappa+c^{2}-1}{c^{2}}, \ \ \ \mu'=\frac{\mu+2c-2}{c}.
\end{equation}
In \cite{boeckx2000} Boeckx  provided a local classification of non-Sasakian $(\kappa,\mu)$-spaces based on the number \eqref{boeckxnumber}, which is an
invariant of a contact metric $(\kappa,\mu)$-structure up to ${\mathcal D}$-homothetic deformations. He proved  that
 two  non-Sasakian  contact  metric $(\kappa,\mu)$-spaces, denoted by
$(M_1,\varphi_1,\xi_1,\eta_1,g_1)$  and $(M_2,\varphi_2,\xi_2,\eta_2,g_2)$, are locally isometric as contact metric manifolds, up to $\mathcal
D$-homothetic deformations, if and only if $I_{M_1}=I_{M_2}$. A geometric interpretation of the invariant $I_{M}$ and of the Boeckx's classification was
recently given in \cite{mino3}.

The standard example of  $(\kappa,\mu)$-spaces is given by the tangent sphere bundle $T_{1}N$ of a manifold of constant curvature $c\neq 1$ endowed with
its standard contact metric structure. In this case $\kappa=c(2-c)$, $\mu=-2c$ and $I_{T_{1}N}=\frac{1+c}{|1-c|}$. Other examples are given by certain Lie
groups defined by Boeckx in \cite{boeckx2000}.

We conclude the subsection by recalling the following formula for the Lie derivative of the operator $h$ in any non-Sasakian $(\kappa,\mu)$-space (cf.
\cite[Lemma 4.5]{mino1})
\begin{equation}\label{lieh}
{\mathcal L}_{\xi}h = (2-\mu)\varphi h + 2(1-\kappa)\varphi.
\end{equation}

\subsection{Paracontact geometry}

An \emph{almost paracontact structure} (cf. \cite{kaneyuki}) on a
$(2n+1)$-dimensional smooth manifold $M$ is given by a
$(1,1)$-tensor field $\w\varphi$, a vector field $\xi$ and a
$1$-form $\eta$ satisfying the following conditions
\begin{enumerate}
  \item[(i)] $\eta(\xi)=1$, \ $\w\varphi^2=I-\eta\otimes\xi$,
  \item[(ii)] the eigendistributions ${\mathcal D}^+$ and ${\mathcal D}^-$ of $\w\varphi$ corresponding to the eigenvalues $1$ and $-1$, respectively, have equal dimension $n$.
\end{enumerate}

As an immediate consequence of the definition, one has that
$\w\varphi\xi=0$, $\eta\circ\w\varphi=0$ and the field of
endomorphisms $\w\varphi$ has constant rank $2n$. As for the almost
contact case, one can consider the almost paracomplex structure on
$M\times\mathbb{R}$ defined by
$\w{J}\bigl(X,f\frac{d}{dt}\bigr)=\bigl(\w{\varphi}X+f\xi,\eta(X)\frac{d}{dt}\bigr)$,
where $X$ is a vector field on $M$ and $f$ a $C^{\infty}$ function
on $M\times\mathbb{R}$. By definition, if $\w{J}$ is integrable the
almost paracontact structure $(\w{\varphi},\xi,\eta)$ is said to be
\emph{normal}. The computation of $\w{J}$ in terms of the tensors of
the almost paracontact structure leads us to define a tensor field
$N_{\w\varphi}$ of type $(1,2)$ given by
$N_{\w\varphi}:=[\w\varphi,\w\varphi]-2d\eta\otimes\xi$. The almost
paracontact structure is then normal if and only if $N_{\w\varphi}$
vanishes identically (cf. \cite{zamkovoy}). Normality in paracontact
geometry has the following geometric interpretation.

\begin{theorem}[\cite{mino4}]\label{normality}
An almost paracontact manifold  $(M,\w\varphi,\xi,\eta)$  is normal if and only if the eigendistributions ${\mathcal D}^{+}$ and ${\mathcal D}^{-}$ of
$\w\varphi$ corresponding to the eigenvalues $\pm 1$ are integrable and the vector field $\xi$ is a foliated vector field with respect to both the
foliations, i.e. $[\xi,X]\in\Gamma({\mathcal D}^{\pm})$ for any $X\in\Gamma({\mathcal D}^{\pm})$.
\end{theorem}

If an almost paracontact manifold is endowed with  a semi-Riemannian
metric $\w g$ such that
\begin{equation}\label{compatibile}
\w g(\w\varphi X,\w\varphi Y)=-\w g(X,Y)+\eta(X)\eta(Y)
\end{equation}
for all  $X,Y\in\Gamma(TM)$, then $(M,\w\varphi,\xi,\eta,\w g)$ is called an \emph{almost paracontact metric manifold}. Notice that any such a
semi-Riemannian metric is necessarily of signature $(n,n+1)$ and the above condition (ii) of the definition of almost paracontact structures is
automatically satisfied. Moreover, as in the almost contact case, from \eqref{compatibile} it follows easily that $\eta=\tilde{g}(\cdot,\xi)$ and
$\w{g}(\cdot,\w\varphi\cdot)=-\w{g}(\w\varphi\cdot,\cdot)$. Hence one defines the \emph{fundamental $2$-form} of the almost paracontact metric manifold by
$\w\Phi(X,Y)=\w{g}(X,\w\varphi Y)$. If $d\eta=\w\Phi$, $\eta$ becomes a contact form and $(M,\w\varphi,\xi,\eta,\w g)$ is said to be a \emph{paracontact
metric manifold}.

On a paracontact metric manifold one defines the tensor field $\w{h}:=\frac{1}{2}{\mathcal L}_{\xi}\w\varphi$. It was proved in \cite{zamkovoy} that $\w h$
is a symmetric operator with respect to $\w{g}$, it anti-commutes with $\w\varphi$ and it vanishes identically if and only if $\xi$ is a Killing vector
field and in such case $(M,\w\varphi,\xi,\eta,\w g)$ is called a \emph{K-paracontact manifold}. Moreover the  identity $\w\nabla\xi=-\w\varphi+\w\varphi\w
h$ holds. A paracontact metric manifold is said to be \emph{integrable}, or \emph{para-CR}, if the following condition is satisfied, for any
$X,Y\in\Gamma(TM)$,
\begin{equation*}
(\w\nabla_{X}\w\varphi)Y=\eta(Y)(X-\w{h}X)-\w{g}(X-\w{h}X,Y)\xi.
\end{equation*}
%
%
A normal paracontact metric manifold is said to be a \emph{paraSasakian manifold}. Also in this context the paraSasakian condition implies the
K-paracontact condition and the converse holds in dimension $3$. In terms of the covariant derivative of $\w\varphi$ the paraSasakian condition may be
expressed by
\begin{equation*}
(\w\nabla_{X}\w\varphi)Y=-\w g(X,Y)\xi+\eta(Y)X.
\end{equation*}
Clearly, for K-paracontact manifolds the notion of integrability coincides with that of being paraSasakian. An equivalent definition of paraSasakian
manifolds is presented in \cite{alek09} in terms of pseudo-Riemannian cones. Standard examples of paraSasakian manifolds are the hyperboloid
\begin{equation*}
\mathbb{H}^{2n+1}_{n+1}(1)=\left\{(x_{0},y_{0},\ldots,x_{n},y_{n})\in\mathbb{R}^{2n+2} \ | \
x_{0}^{2}+\ldots+x_{n}^{2}-y_{0}^{2}-\ldots-y_{n}^{2}=1\right\}
\end{equation*}
and the hyperbolic Heisenberg group ${\mathcal H}^{2n+1}=\mathbb{R}^{2n}\times\mathbb{R}$ with the structures defined in \cite{ivanov}. Furthermore, let us
recall that a notion of $\eta$-Einstein metric and ${\mathcal D}_c$-homothetic deformation can be introduced also in paracontact metric geometry
(\cite{zamkovoy}). The definition is the same as the one given in \eqref{etaeinstein1} and \eqref{deformation} for contact Riemannian manifolds, with the
only change that the constant of homothety $c$ can be now any non-zero real number since the metric does not need to be positive definite.

\section{The main results}\label{second}
There is a strict relationship between the theory of contact metric $(\kappa,\mu)$-spaces and of paracontact geometry, as shown in \cite{mino2} and
\cite{mino1}. In fact, given a non-Sasakian contact metric $(\kappa,\mu)$-space $(M,\varphi,\xi,\eta,g)$, one can define  canonically two \emph{integrable}
paracontact metric structures on $M$, $(\w\varphi_{1},\xi,\eta,\w{g}_{1})$ and $(\w\varphi_{2},\xi,\eta,\w{g}_{2})$, which are compatible with the same
underlying contact form and Reeb vector field as the $(\kappa,\mu)$-space $M$. They are defined by
\begin{gather*}
\w\varphi_{1}:=\frac{1}{\sqrt{1-\kappa}}\varphi h, \ \ \ \
\w{g}_{1}:=\frac{1}{\sqrt{1-\kappa}}g(\cdot, h
\cdot)+\eta\otimes\eta, 
 \\
\w\varphi_{2}:=\frac{1}{\sqrt{1-\kappa}}h, \ \ \ \
\w{g}_{2}:=\frac{1}{\sqrt{1-\kappa}}g(\cdot,\varphi h
\cdot)+\eta\otimes\eta. 
\end{gather*}
The curvature tensor fields of such paracontact metric structures, in turn, satisfy a  nullity-like condition
\begin{equation*}
\w{R}_{\alpha}(X,Y)\xi = \w\kappa_{\alpha}(\eta(Y)X-\eta(X)Y) +
\w\mu_{\alpha}(\eta(Y)\w{h}_{\alpha}X-\eta(X)\w{h}_{\alpha}Y),
\end{equation*}
where $\w\kappa_{1}=\left(1-\frac{\mu}{2}\right)-1$, $\w\mu_{1}=2\left(1-\sqrt{1-\kappa}\right)$ and
$\w\kappa_{2}=\kappa-2+\left(1-\frac{\mu}{2}\right)^2$, $\w\mu_{2}=2$. By \eqref{lieh} one can prove that $\w{h}_{1}=-I_{M} h$. Hence, being integrable,
the paracontact metric structure $(\w\varphi_{1},\xi,\eta,\w{g}_{1})$ is paraSasakian if and only if $I_{M}=0$. Whereas for no value of $\kappa$ and $\mu$
$(\w\varphi_{2},\xi,\eta,\w{g}_{2})$ is paraSasakian. Now we prove a much stronger result.

\begin{theorem}\label{main1}
Let $(M,\varphi,\xi,\eta,g)$ be a non-Sasakian contact metric $(\kappa,\mu)$-space such that $I_M \neq \pm 1$.
\begin{enumerate}
  \item[(i)] If $|I_{M}|>1$, then $M$ admits a Sasakian structure,
  compatible with the contact form $\eta$, given by
  \begin{equation}\label{sasaki}
\bar{\phi}:=\epsilon \frac{1}{(1-\kappa)\sqrt{(2-\mu)^2-4(1-\kappa)}}{\mathcal L}_{\xi}h \circ h, \qquad \bar{g}:=-d\eta(\cdot,\bar\phi\cdot) +
\eta\otimes\eta,
\end{equation}
where
\begin{equation}\label{epsilon}
\epsilon:=\left\{
                                   \begin{array}{cl}
                                     1 & \hbox{if $I_M>1$} \\
                                     -1 & \hbox{if $I_M<-1$ }
                                   \end{array}
                                 \right.
\end{equation}

\item[(ii)] If $|I_{M}|<1$, then $M$ admits a paraSasakian structure,
  compatible with the contact form $\eta$, given by
  \begin{equation}\label{parasasaki}
  \w\phi:=\frac{1}{(1-\kappa)\sqrt{4(1-\kappa)-(2-\mu)^2}}{\mathcal
  L}_{\xi}h \circ h,  \ \ \
  \w{g}:=d\eta(\cdot,\w\phi\cdot)+\eta\otimes\eta.
  \end{equation}
\end{enumerate}
\end{theorem}
\begin{proof}
By formula \eqref{lieh}, using \eqref{hformula1} and the
anti-commutativity of $\varphi$ and $h$, one has that
\begin{equation}\label{formulasasakian1}
{\mathcal L}_{\xi}h\circ h=(1-\kappa)((2-\mu)\varphi+2\varphi h).
\end{equation}
Hence
\begin{equation*}
({\mathcal L}_{\xi}h \circ h)^2=(2-\mu)^{2}(1-\kappa)^{2}\varphi^{2}-4(1-\kappa)^{2}\varphi^{2}h^{2}=(1-\kappa)^{2}((2-\mu)^{2}-4(1-\kappa))\varphi^{2}.
\end{equation*}
Therefore $({\mathcal L}_{\xi}h \circ h)^2 = \lambda^{4}\alpha (-I+\eta\otimes\xi)$, where we recall that $\lambda=\sqrt{1-\kappa}$ and we have put
$\alpha:=(2-\mu)^{2}-4(1-\kappa)$. Consequently, if $\alpha>0$ then $\bar\phi:=\epsilon \frac{1}{\lambda^2\sqrt{\alpha}}{\mathcal L}_{\xi}h \circ h$
defines an almost contact structure and if $\alpha<0$ then $\w\phi:=\frac{1}{\lambda^2\sqrt{-\alpha}}{\mathcal L}_{\xi}h \circ h$ satisfies the first of
the two conditions defining an almost paracontact structure. Notice that, since $\kappa\leq 1$, $\alpha>0$ if and only if $|I_{M}|>1$.

\textbf{(i)} Let us assume that $|I_{M}|>1$. We firstly show that $\bar\phi$ is a normal almost contact structure. Just by using the definition of the
tensor field $N_{\phi}$, one can prove that for any almost contact structure $(\phi,\xi,\eta)$ (same contact form and Reeb vector field as
$(\varphi,\xi,\eta,g)$) the following identity holds
\begin{equation}\label{intermedia1}
N_{\phi}(X,Y)=(\nabla_{\phi X}\phi)Y-(\nabla_{\phi
Y}\phi)X+(\nabla_{X}\phi)\phi Y-(\nabla_{Y}\phi)\phi
X-\eta(Y)\nabla_{X}\xi+\eta(X)\nabla_{Y}\xi,
\end{equation}
where $\nabla$ is the Levi-Civita connection of the Riemannian metric $g$. By \eqref{formulasasakian1} we have that
\begin{align}
\bar\phi &=\epsilon \frac{1}{\sqrt{\alpha}}((2-\mu)\varphi+2\varphi
h),\label{sasaki-phi}
\end{align}
so that by \eqref{intermedia1} we find
\begin{align}\label{intermedia2}
N_{\bar\phi}(X,Y)=&\frac{(2-\mu)^2}{\alpha}(\nabla_{\varphi X}\varphi)Y  + \frac{2(2-\mu)}{\alpha}(\nabla_{\varphi h X}\varphi)Y + \frac{2(2-\mu)}{\alpha}(\nabla_{\varphi X}\varphi h)Y + \frac{4}{\alpha}(\nabla_{\varphi h X}\varphi h)Y \nonumber \\
& -\frac{(2-\mu)^2}{\alpha}(\nabla_{\varphi Y}\varphi)X - \frac{2(2-\mu)}{\alpha}(\nabla_{\varphi h Y}\varphi)X - \frac{2(2-\mu)}{\alpha}(\nabla_{\varphi Y}\varphi h)X - \frac{4}{\alpha}(\nabla_{\varphi h Y}\varphi h)X \nonumber \\
 &+\frac{(2-\mu)^2}{\alpha}(\nabla_{X}\varphi)\varphi Y + \frac{2(2-\mu)}{\alpha}(\nabla_{X}\varphi)\varphi h Y + \frac{2(2-\mu)}{\alpha}(\nabla_{X}\varphi h)\varphi Y + \frac{4}{\alpha}(\nabla_{X}\varphi h)\varphi hY  \\
&-\frac{(2-\mu)^2}{\alpha}(\nabla_{Y}\varphi)\varphi X - \frac{2(2-\mu)}{\alpha}(\nabla_{Y}\varphi)\varphi h X - \frac{2(2-\mu)}{\alpha}(\nabla_{Y}\varphi h)\varphi X - \frac{4}{\alpha}(\nabla_{Y}\varphi h)\varphi hX \nonumber \\
 & -\eta(Y)\nabla_{X}\xi+\eta(X)\nabla_{Y}\xi.\nonumber
\end{align}
Then taking \eqref{acca2}, \eqref{covariant1} and \eqref{covariant2}
into account, and using \eqref{hformula1} and the
anti-commutativity of $\varphi$ and $h$, after very long
computations one can prove that \eqref{intermedia2} reduces to
\begin{align}
N_{\bar\phi}&(X,Y)=\left(1-\frac{(2-\mu)^2}{\alpha}+\frac{4\lambda^2(2-\mu)}{\alpha}-\frac{4\lambda^2}{\alpha}+\frac{4\lambda^2\mu}{\alpha}\right)\left(\eta(Y)\varphi X - \eta(X)\varphi Y\right) \nonumber \\
&\quad+\left(1+\frac{(2-\mu)^2}{\alpha}-\frac{4(2-\mu)}{\alpha}+\frac{4\lambda^2}{\alpha}+\frac{2(2-\mu)\mu}{\alpha}\right)\left(\eta(Y)\varphi
h X - \eta(X)\varphi h Y\right). \label{intermedia3}
\end{align}
By substituting the values of $\lambda$ and $\alpha$ in
\eqref{intermedia3} one can check that in fact $N_{\bar\phi}(X,Y)=0$.
\ Next we prove that the tensor $\bar{g}$ in \eqref{sasaki} defines a
Riemannian metric compatible with the almost contact structure
$(\bar{\phi},\xi,\eta)$ and such that
$d\eta=\bar{g}(\cdot,\bar{\varphi}\cdot)$. Firstly notice that from
\eqref{sasaki} and \eqref{formulasasakian1} it follows that
\begin{align}
\bar{g}(X,Y)&=\epsilon \frac{1}{\sqrt{(2-\mu)^{2}-4(1-\kappa)}}\left((2-\mu)g(X,Y)+2g(X,hY)\right) \nonumber \\
&\quad + \left(1-\epsilon \frac{2-\mu}{\sqrt{(2-\mu)^{2}-4(1-\kappa)}}\right)\eta(X)\eta(Y), \label{formulametrica1}
\end{align}
where $\epsilon$ is given by \eqref{epsilon}. Since $h$ is a symmetric operator with respect to the Riemannian metric $g$, by \eqref{formulametrica1} we
see immediately that $\bar{g}$ is a symmetric tensor. Next we prove that it is positive definite. First, by the very definition of $\bar{g}$, one has that
$\bar{g}(\xi,\xi)=1$. Moreover, for any non-zero tangent vector field $X\in\Gamma(\mathcal{D})$,   one has by \eqref{formulasasakian1} that
\begin{equation*}
\bar{g}(X,X)=\epsilon
\frac{1}{\sqrt{(2-\mu)^{2}-4(1-\kappa)}}\left((2-\mu)g(X,X)+2g(X,hX)\right).
\end{equation*}
We can distinguish two cases:  $X\in\Gamma(\mathcal{D}_{h}(\lambda))$ or $X\in\Gamma(\mathcal{D}_{h}(-\lambda))$. In the first case we obtain
\begin{equation*}
\bar{g}(X,X)=\epsilon
\frac{(2-\mu)+2\sqrt{1-\kappa}}{\sqrt{(2-\mu)^{2}-4(1-\kappa)}}g(X,X).
\end{equation*}
The above formula implies that $\bar{g}(X,X)>0$, since $(2-\mu)+2\sqrt{1-\kappa}$ is always positive when ${I_M}>1$ and negative when $I_M<-1$, the same as
$\epsilon$. The other case $X\in\Gamma(\mathcal{D}_{h}(-\lambda))$ is similar. It remains to be seen that $\bar{g}$ is an \emph{associated metric}, that is
$d\eta=\bar{g}(\cdot,\bar{\phi}\cdot)$. Indeed by the definition of $\bar{g}$ and by the property of $\bar\phi$ of defining an almost contact structure we
have $\bar{g}(X,\bar{\phi}Y)=-d\eta(X,\bar{\phi}^{2}Y)+\eta(X)\eta(\bar\phi Y)=-d\eta(X,-Y+\eta(Y)\xi)=d\eta(X,Y)+\eta(Y)d\eta(X,\xi)=d\eta(X,Y)$.

\textbf{(ii)} Now let us assume that $|I_M|<1$. In this case $\w\phi$ is given by
\begin{equation}\label{parasasaki-phi}
\w\phi=\frac{1}{\sqrt{-\alpha}}((2-\mu)\varphi+2\varphi h)
\end{equation}
and the tensor $\w{g}$ by
\begin{align}
\w{g}(X,Y)&=-\frac{1}{\sqrt{-\alpha}}\left((2-\mu)g(X,Y)+2g(X,hY)\right)
+\left(1+\frac{2-\mu}{\sqrt{-\alpha}}\right)\eta(X)\eta(Y).
\label{formulametrica2}
\end{align}
Equation \eqref{formulametrica2} easily implies that $\w{g}$ is a symmetric tensor. Moreover, one can prove as in the case (i) that
$d\eta(X,Y)=\w{g}(X,\w\phi Y)$ for any $X,Y\in\Gamma(TM)$. This last condition, together with $\w\phi^{2}=I-\eta\otimes\xi$, implies that $\w{g}(\w\phi
X,\w\phi Y)=-\w{g}(X,Y)+\eta(X)\eta(Y)$. Thus $\w{g}$ is a semi-Riemannian metric of signature $(n,n+1)$ and $(\w\phi,\xi,\eta,\w{g})$ is a paracontact
metric manifold. So, once we have proved that the structure $(\w\phi,\xi,\eta)$ is normal, we would have proved that $(M,\w\phi,\xi,\eta,\w{g})$ is a
paraSasakian manifold. In order to do that, we will check that the almost paracontact structure $(\w\phi,\xi,\eta)$ satisfies the conditions of Theorem
\ref{normality}. First we prove that the eigendistributions ${\mathcal D}^{\pm}$ are integrable. Let $X,X'\in\Gamma({\mathcal D}^{+})$, i.e. $\w\phi X=X$,
$\w\phi X'=X'$  By the definition of $\w\phi$ and formulas
\eqref{integrabile1} and \eqref{covariant2}, we get
\begin{align*}
\w\phi[X,X']=&\nabla_{X}\w\phi X' - \nabla_{X'}\w\phi X - \frac{2-\mu}{\sqrt{-\alpha}}g(X+hX,X')\xi + \frac{2-\mu}{\sqrt{-\alpha}}g(X'+hX',X)\xi\\
&-\frac{2}{\sqrt{-\alpha}}\left(g(X,hX')\xi+(1-\kappa)g(X,X')\xi\right)+ \frac{2}{\sqrt{-\alpha}}\left(g(X',hX)\xi+(1-\kappa)g(X',X)\xi\right)\\
=&\nabla_{X}X'-\nabla_{X'}X =[X,X'].
\end{align*}
Hence $[X,X']\in\Gamma({\mathcal D}^+)$ and we conclude that  ${\mathcal D}^+$ is an integrable distribution.  It remains to be seen that $\xi$ is a
foliated vector field with respect to the foliation defined by  ${\mathcal D}^+$, i.e. $[\xi,X]\in\Gamma({\mathcal D}^+)$ for any $X\in\Gamma({\mathcal
D}^+)$. By using \eqref{acca2}, $\nabla_{\xi}\varphi=0$ and $\nabla_{\xi}\varphi h=\mu h$  we have, arguing as before,
\begin{align}
\w\phi[\xi,X]&=\nabla_{\xi}\left(\frac{2-\mu}{\sqrt{-\alpha}}\varphi X+\frac{2}{\sqrt{-\alpha}}\varphi h X\right) - \frac{\mu}{\sqrt{-\alpha}}hX +
\frac{\mu-2k}{\sqrt{-\alpha}}X \label{intermedia4}.
\end{align}
On the other hand, $\frac{2-\mu}{\sqrt{-\alpha}}\varphi X+\frac{2}{\sqrt{-\alpha}}\varphi h X=\w\phi X$, and one can check that $\nabla_{\w\phi
X}\xi=-\varphi\w\phi X - \varphi h\w\phi X = \frac{\mu}{\sqrt{-\alpha}}hX - \frac{\mu-2k}{\sqrt{-\alpha}}X$. So that \eqref{intermedia4} yields
$\w\phi[\xi,X]=\nabla_{\xi}\w\phi X-\nabla_{\w\phi X}\xi=\nabla_{\xi}X-\nabla_{X}\xi=[\xi,X]$, that is $[\xi,X]\in\Gamma({\mathcal D}^{+})$. The same
arguments also work for ${\mathcal D}^{-}$. Therefore, by Theorem \ref{normality} we conclude that the paracontact metric structure $(\w\phi,\xi,\eta,g)$
is normal and hence paraSasakian.
\end{proof}

\begin{remark}
One can see that the  metrics stated in Theorem \ref{main1} extend the Riemannian metrics introduced, in a different context, in \cite{mino2} and
\cite{mino1}.
\end{remark}

\begin{remark}\label{deformation2}
It is interesting to notice that the metrics $\bar{g}$ and $\w{g}$
are invariant under $\mathcal D$-homothetic deformations. Indeed let
us apply a ${\mathcal D}_c$-homothetic deformation
\eqref{deformation} to the non-Sasakian contact metric
$(\kappa,\mu)$-space $(M,\varphi,\xi,\eta,g)$, for some $c>0$.
Then we obtain a new contact metric $(\kappa',\mu')$-structure
$(\varphi',\xi',\eta',g')$ on $M$, with $\kappa'$ and $\mu'$ given
by \eqref{deformation1} and the same Boeckx invariant as
$(\varphi,\xi,\eta,g)$. Thus  $(\varphi',\xi',\eta',g')$ in turn
satisfies  the assumptions of Theorem \ref{main1} and admits a
Sasakian structure  $(\bar{\phi}',\xi',\eta',\bar{g}')$ or a
paraSasakian structure $(\w{\phi}',\xi',\eta',\w{g}')$ according
to the circumstance that $|I_{M}|>1$ or $|I_{M}|<1$, respectively.
Let us assume for instance that $|I_{M}|>1$. Since $h'=\frac{1}{c}h$
and a straightforward computation shows that
$\alpha'=(2-\mu')^{2}-4(1-\kappa')=\frac{1}{c^{2}}\alpha$, from
\eqref{sasaki-phi} it follows that
\begin{equation*}
\bar{\phi}'=\epsilon\frac{1}{\sqrt{\alpha'}}\left((2-\mu')\varphi'+2\varphi'h'\right)=\epsilon\frac{1}{\sqrt{\alpha}}\left((2-\mu)\varphi+2\varphi
 h\right)=\bar{\phi}.
\end{equation*}
Moreover, due to $\bar{g}'(\cdot,\bar{\phi}'\cdot)=d\eta'$, one easily proves  that $\bar{g}'=c \bar{g} + c(c-1)\eta\otimes\eta$. Thus the Sasakian
structure $(\bar{\phi}',\xi',\eta',\bar{g}')$, associated via Theorem \ref{main1} to the $(\kappa,\mu)$-space $(M,\varphi',\xi',\eta',g')$, is nothing but
the structure ${\mathcal D}_{c}$-homothetic to the Sasakian structure $(\bar{\phi},\xi,\eta,\bar{g})$. The same arguments work in the case $|I_{M}|<1$.
\end{remark}

We will now calculate the curvature tensor of the Sasakian and
paraSasakian manifolds that appear in the previous theorem, which we
will show to be always $\eta$-Einstein.

\begin{theorem}\label{main2}
Let $(M,\varphi,\xi,\eta,g)$ be a non-Sasakian contact metric $(\kappa,\mu)$-space such that $I_M \neq \pm 1$.
\begin{enumerate}
  \item[(i)] If $|I_{M}|>1$ then the Levi-Civita connection of $(M,\bphi,\xi,\eta,\bar{g})$, the Sasakian manifold  defined in  \eqref{sasaki}, relates to the original one in the following way:
  \begin{equation}\label{sasaki-nabla}
  \begin{aligned}
  \onabla_X Y=&\nabla_X Y+\left( 1-\epsilon\frac{2-\mu}{\sqrt{\alpha}} \right)(\eta(Y)\varphi X+\eta(X)\varphi Y)\\
  &+\left( 1-\epsilon\frac{2}{\sqrt{\alpha}} \right) (\eta(Y)\varphi h X+\eta(X)\varphi h Y)-g(X,\varphi h Y)\xi,
  \end{aligned}
  \end{equation}
  where $\epsilon$ is given by \eqref{epsilon}. Moreover, its curvature tensor has the form
    \begin{equation}\label{sasaki-R}
    \begin{aligned}
    \overline{R}(X,Y)Z&=\epsilon\frac{\sqrt{\alpha}}{2} (\bar{g}(Y,Z)X-\bar{g}(X,Z)Y)\\
    &+\left(\epsilon\frac{\sqrt{\alpha}}{2}-1 \right) (\bar{g}(X,\bphi Z)\bphi Y-\bar{g}(Y,\bphi Z)\bphi X+2\bar{g}(X,\bphi Y)\bphi Z)\\
    &+\left(\epsilon\frac{\sqrt{\alpha}}{2}-1 \right) (\eta(X)\eta(Z)Y-\eta(Y)\eta(Z)X+\bar{g}(X,Z)\eta(Y) \xi-\bar{g}(Y,Z)\eta(X) \xi)\\
    &+\epsilon\frac{\sqrt{\alpha}}{2(1-\kappa)} \left( \bar{g}(h Y,Z)h X-\bar{g}(h X,Z)hY-\bar{g}(hY,\bphi Z)\bphi hX+\bar{g}(hX,\bphi Z)\bphi h Y \right).
    \end{aligned}
    \end{equation}
  The Ricci tensor of $M^{2n+1}$ is
    \begin{equation}\label{sasaki-ricci}
    \overline{\emph{Ric}}=(\epsilon n\sqrt{\alpha}-2)\bar{g}+(-\epsilon n\sqrt{\alpha} + 2n+2)\eta \otimes \eta
    \end{equation}
  and therefore the Sasakian structure is always $\eta$-Einstein.
  \item[(ii)] If $|I_{M}|<1$ then the paraSasakian structure $(\wphi,\xi,\eta,\bar{g})$ defined in \eqref{parasasaki} has the following Levi-Civita connection:
  \begin{equation}\label{parasasaki-nabla}
  \begin{aligned}
  \wnabla_X Y=&\nabla_X Y+\left( 1-\frac{2-\mu}{\sqrt{-\alpha}} \right)(\eta(Y)\varphi X+\eta(X)\varphi Y)\\
  &+\left( 1-\frac{2}{\sqrt{-\alpha}} \right) (\eta(Y)\varphi h X+\eta(X)\varphi h Y)-g(X,\varphi h Y)\xi.
  \end{aligned}
  \end{equation}
  Furthermore, the curvature tensor can be written as
  \begin{equation}\label{parasasaki-R}
    \begin{aligned}
    \w{R}(X,Y)Z&=-\frac{\sqrt{-\alpha}}{2} (\w{g}(Y,Z)X-\w{g}(X,Z)Y)\\
    &+\left( \frac{\sqrt{-\alpha}}{2}-1 \right) (\w{g}(X,\wphi Z)\wphi Y-\w{g}(Y,\wphi Z)\wphi X+2\w{g}(X,\wphi Y)\wphi Z)\\
    &-\left( \frac{\sqrt{-\alpha}}{2}-1 \right) (\eta(X)\eta(Z)Y-\eta(Y)\eta(Z)X+\w{g}(X,Z)\eta(Y) \xi-\w{g}(Y,Z)\eta(X) \xi)\\
    &-\frac{\sqrt{-\alpha}}{2(1-\kappa)} \left( \w{g}(h Y,Z)h X-\w{g}(h X,Z)hY+\w{g}(hY,\wphi Z)\wphi hX-\w{g}(hX,\wphi Z)\wphi h Y \right).
    \end{aligned}
    \end{equation}
    The Ricci tensor of $M^{2n+1}$ is
    \begin{equation}\label{parasasaki-ricci}
    \w{\emph{Ric}}=(-n\sqrt{-\alpha}+3)\w{g}+(n\sqrt{-\alpha} - 2n-3)\eta \otimes \eta
    \end{equation}
  and therefore the paraSasakian structure is always $\eta$-Einstein.
\end{enumerate}
\end{theorem}
\begin{proof}
\textbf{(i)} Bearing in mind that $\onabla$ is a Levi-Civita connection, we can use Koszul's formula, $\nabla g=0$ and formulas \eqref{covariant1}
and \eqref{formulametrica1} to obtain
\begin{equation}\label{nabla1-intermedia}
\begin{split}
2 \bar{g} (\onabla_X Y,Z)=& 2\bar{g}(\nabla_X Y,Z)+2\left( 1+\epsilon\frac{\mu-2\kappa}{\sqrt{\alpha}} \right)(\eta(Y)g(X,\varphi Z)+\eta(X)g(Y,\phi Z))\\
&-2\eta(Z)g(X,\phi h Y) -\epsilon\frac{2\mu}{\sqrt{\alpha}}(\eta(X)g(Y,\phi h Z)+\eta(Y)g(X,\phi h Z)).
\end{split}
\end{equation}
We recall now that $\eta$ is the contact form of both structures, so
$g(X,\varphi Y)=d\eta(X,Y)=\bar{g}(X,\bphi Y)$ and $g(Y,\phi h
Z)=-\bar{g}(h Y, \bphi Z)$ is deduced. Substituting both formulas in
\eqref{nabla1-intermedia}, it follows that
\begin{equation*}
\begin{split}
\onabla_X Y=& \nabla_X Y+\left( 1+\epsilon\frac{\mu-2\kappa}{\sqrt{\alpha}} \right)(\eta(Y)\bphi X+\eta(X)\bphi Y)-g(X,\phi h
Y)\xi-\epsilon\frac{\mu}{\sqrt{\alpha}}(\eta(X)\bphi h Y+\eta(Y)\bphi h X).
\end{split}
\end{equation*}
This last equation together with \eqref{sasaki-phi} gives us \eqref{sasaki-nabla}. As $\overline{R}$ is a Riemannian curvature tensor, using equations
\eqref{integrabile1} and \eqref{covariant2} we obtain after long computations that
\begin{align}\label{intermedia5}
\overline{R}(X,Y)Z=&R(X,Y)Z+\left(1-\epsilon\frac{2-\mu}{\sqrt{\alpha}}\right) (g(X,\varphi Z)\varphi Y-g(Y,\varphi Z)\varphi X\\
\quad &+2 g(X,\varphi Y)\varphi Z) +g(Y,\varphi h Z)\varphi X-g(X,\varphi h Z)\varphi Y \nonumber\\
\quad &+\left(1-\epsilon\frac{2}{\sqrt{\alpha}} \right) (g(\varphi Y,Z)\varphi h X-g(\varphi X,Z)\varphi h Z+2 g(X,\varphi Y)\varphi h Z)\nonumber
\end{align}
\begin{align}
\quad &+(1-\kappa)(\eta(Y)\eta(Z)X-\eta(X)\eta(Z)Y) +\left(\kappa-\epsilon\frac{2-\mu}{\sqrt{\alpha}} \right) (\eta(Y)g(X,Z)\xi-\eta(X)g(Y,Z)\xi)
\nonumber\\
\quad &-\mu(\eta(Y)\eta(Z)hX-\eta(X)\eta(Z)hY)+\left(\mu-\epsilon\frac{2}{\sqrt{\alpha}}\right)(\eta(Y)g(X,hZ)\xi-\eta(X)g(Y,h Z)\xi. \nonumber
\end{align}
On the other hand, we know from \cite{boeckx2000} the writing of the curvature tensor of a non-Sasakian $(\kappa,\mu)$-space, which substituting in \eqref{intermedia5} gives us
\begin{equation*}
\begin{aligned}
\overline{R}(X,Y)Z=&\left(1-\frac{\mu}{2} \right)(g(Y,X)X-g(X,Z)Y)+\left(1-\epsilon\frac{2-\mu}{\sqrt{\alpha}}-\frac{\mu}{2}\right)\bigl(g(X,\varphi Z)\varphi Y-{g}(Y,\varphi Z)\varphi X\\
\quad&+2{g}(X,\varphi Y)\varphi Z+g(X,Z)\eta(Y)\xi -{g}(Y,Z)\eta(X) \xi\bigr)+\left(1-\epsilon\frac{2}{\sqrt{\alpha}}\right)\bigl(g(\varphi Y,Z)\varphi h X\\
\quad &-g(\varphi X,Z)\varphi h Z+2 g(X,\varphi Y)\varphi h Z +{g}(hX,Z)\eta(Y) \xi-{g}(hY,Z)\eta(X) \xi\bigr)   \\
\quad &+ \frac{1-\frac{\mu}{2}}{1-\kappa}  (g(hY,X)h X-g(hX,Z)h Y+g(\varphi hY,X)\varphi h X-g(\varphi hX,Z)\varphi h Y) \\
\quad &  -\frac{\mu}{2}(\eta(X)\eta(Z)Y-\eta(Y)\eta(Z)X)+ g(Y,X)h X-g(X,Z)h Y+g(h Y,X)X-g(h X,Z)Y \\
\quad &+ \eta(X)\eta(Z)hY-\eta(Y)\eta(Z)hX+g(Y,\varphi h Z)\varphi X-g(X,\varphi h Z)\varphi Y
\end{aligned}
\end{equation*}
Thus by using \eqref{sasaki-phi} and \eqref{formulametrica1} we get \eqref{sasaki-R}.  We will finally compute the Ricci tensor, for which we will
construct an orthonormal basis with respect to the structure $(\bphi,\xi,\eta,\bar{g})$. Let us take a $\varphi$-basis $\{X_i,Y_i=\varphi X_i,\xi\}$,
$i=1,\ldots,n$, such that  $X_i \in {\mathcal D}_h(\lambda)$ (and therefore $Y_i \in {\mathcal D}_h(-\lambda)$), which always exists because
$(M,\phi,\xi,\eta,{g})$ is a non-Sasakian $(\kappa,\mu)$-space. Then we define $\overline{X}_i=\frac{1}{\sqrt{\gamma}} X_i$ and
$\overline{Y}_i=\sqrt{\gamma} Y_i$, where $\gamma=\epsilon\frac{2-\mu+2\lambda}{\sqrt{\alpha}}$. Notice that $\gamma>0$ because $|I_M|>1$ and that we can
write $\frac{1}{\gamma}=\epsilon\frac{2-\mu-2\lambda}{\sqrt{\alpha}}$. Thus $\{\overline{X}_i,\overline{Y}_i,\xi\}$ is a $\bphi$-basis with $\overline{X}_i
\in {\mathcal D}_h(\lambda)$ and $\overline{Y}_i \in {\mathcal D}_h(-\lambda)$.  By the definition of the Ricci tensor, we have that
\begin{align}
\overline{\textrm{Ric}}(X,Y)&=\sum_{i=1}^n
(\bar{g}(\overline{R}(X,\overline{X}_i)\overline{X}_i,Y)+\bar{g}(\overline{R}(X,\overline{Y}_i)\overline{Y}_i,Y))+\bar{g}(\overline{R}(X,\xi)\xi,Y).
\label{intermedia-ricci}
\end{align}
Using formula \eqref{sasaki-R}, we can now compute
\begin{align*}
\bar{g}(\overline{R}(X,\overline{X}_i)\overline{X}_i,Y)&=\epsilon\frac{\sqrt{\alpha}}{2} \bar{g}(X,Y)-\left(\epsilon\frac{\sqrt{\alpha}}{2}-1 \right) \eta(X)\eta(Y)+\epsilon\frac{\sqrt{\alpha}}{2\lambda} \bar{g}(hX,Y)\\
&\quad-\epsilon\sqrt{\alpha} \bar{g}(X,\overline{X}_i)\bar{g}(\overline{X}_i,Y)+(\epsilon\sqrt{\alpha}-3)
\bar{g}(X,\overline{Y}_i)\bar{g}(\overline{Y}_i,Y),\\
\bar{g}(\overline{R}(X,\overline{Y}_i)\overline{Y}_i,Y)&=\epsilon\frac{\sqrt{\alpha}}{2} \bar{g}(X,Y)-\left(\epsilon\frac{\sqrt{\alpha}}{2}-1 \right)
\eta(X)\eta(Y)-\epsilon\frac{\sqrt{\alpha}}{2 \lambda} \bar{g}(hX,Y)\\
 &\quad+(\epsilon\sqrt{\alpha}-3)\bar{g}(X,\overline{X}_i)\bar{g}(\overline{X}_i,Y)-\epsilon\sqrt{\alpha}
\bar{g}(X,\overline{Y}_i)\bar{g}(\overline{Y}_i,Y),\\
\bar{g}(\overline{R}(X,\xi)\xi,Y)&=\bar{g}(X,Y)-\eta(X)\eta(Y),
\end{align*}
which substituting in \eqref{intermedia-ricci} gives us
\eqref{sasaki-ricci}.

\textbf{(ii)} The proof of this case is long but analogous to the previous one, bearing in mind that $\alpha$ changes its sign. We also have to take into
account the causal character of the vector fields of the orthonormal basis constructed to compute the formula of the Ricci tensor. Indeed, we need to take
a $\varphi$-basis $\{X_i,Y_i=\varphi X_i,\xi\}$, $i=1,\ldots,n$, such that  $X_i \in {\mathcal D}_h(\lambda)$ and $Y_i \in {\mathcal D}_h(-\lambda)$. We
then define $\w{X}_i=\frac{1}{\sqrt{\gamma}} X_i$ and $\w{Y}_i=\sqrt{\gamma} Y_i$, where $\gamma=\frac{2-\mu+2\lambda}{\sqrt{-\alpha}}$. We know that
$\gamma>0$ because $|I_M|<1$ and that $\frac{1}{\gamma}=-\frac{2-\mu-2\lambda}{\sqrt{-\alpha}}$. Therefore, $\{\w{X}_i,\w{Y}_i,\xi\}$ is a $\wphi$-basis
with $\w{X}_i \in {\mathcal D}_h(\lambda)$ and $\w{Y}_i \in {\mathcal D}_h(-\lambda)$. Let us notice that $\w{g}(\w{X}_i,\w{X}_i)=-1$ and that
$\w{g}(\w{X}_i,\w{X}_i)=1$ for all $i\in\left\{1,\ldots,n\right\}$.  Thus, by the definition of the Ricci tensor,
\begin{equation*}
\w{\textrm{Ric}}(X,Y)=\sum_{i=1}^{n}(-\w{g}(\w{R}(X,\w{X}_i)\w{X}_i,Y)+\w{g}(\w{R}(X,\w{Y}_i)\w{Y}_i,Y))+\w{g}(\w{R}(X,\xi)\xi,Y)
\end{equation*}
and applying  \eqref{parasasaki-R} gives us \eqref{parasasaki-ricci}.
\end{proof}

\begin{remark}
It is worth noticing that the first three tensors that appear in
equations \eqref{sasaki-R} and \eqref{parasasaki-R} are well-known
because a Sasakian space form $(M,\phi,\xi,\eta,g)$ of constant
$\phi$-sectional curvature $c$ has the following curvature tensor (cf.
\cite{blairbook}):
\begin{equation*}
    \begin{aligned}
    {R}(X,Y)Z&=\frac{c+3}{4} \bigl(g(Y,X)X-g(X,Z)Y\bigr) +\frac{c-1}{4} \bigl(g(X,\phi Z)\phi Y-{g}(Y,\phi Z)\phi X+2{g}(X,\phi Y)\phi Z\bigr)\\
    & \quad +\frac{c-1}{4} \bigl(\eta(X)\eta(Z)Y-\eta(Y)\eta(Z)X+{g}(X,Z)\eta(Y) \xi-{g}(Y,Z)\eta(X) \xi\bigr)
    \end{aligned}
    \end{equation*}
These tensors are sometimes denoted (cf. \cite{alfonso-str}) as
${R}_1$, ${R}_2$ and ${R}_3$ and the writing of the curvature tensor
simplified to
$${R}=\frac{c+3}{4} {R}_1+\frac{c-1}{4} {R}_2+\frac{c-1}{4} {R}_3.$$
\end{remark}

By the previous theorem, we can now compute the sectional curvature:
\begin{corollary}
Let $(M,\varphi,\xi,\eta,g)$ be a non-Sasakian contact metric
$(\kappa,\mu)$-space such that $I_M \neq \pm 1$.
\begin{enumerate}
  \item[(i)] If $|I_{M}|>1$, the sectional curvature of the  Sasakian metric on $M$,  defined by  \eqref{sasaki}, is:
  \begin{equation*}
  \overline{K}(X,Y)=\left\{ \begin{array}{ll}
  1 & \mbox{if } X=\xi \mbox{ or } Y=\xi,\\
  \epsilon\sqrt{\alpha}   & \mbox{if } X,Y \in \Gamma({\mathcal D}_{h}(\lambda))   \mbox{ or } X,Y \in \Gamma({\mathcal D}_{h}(-\lambda)),\\
  (\epsilon\sqrt{\alpha}-3) \bar{g} (X,\bphi Y)^2   & \mbox{if } X \in \Gamma({\mathcal D}_{h}(\pm\lambda)), Y \in \Gamma({\mathcal D}_{h}(\mp\lambda)),
  \end{array} \right.
  \end{equation*}
  where $X$ and $Y$ are mutually orthogonal, unit vector fields with respect to $\bar{g}$.  In particular, if  $X \in \Gamma({\mathcal D}_{h}(\lambda))$    or  $X \in
\Gamma({\mathcal D}_{h}(-\lambda))$, the $\bphi$-sectional curvature is given by
\begin{equation*}
  \overline{K}(X,\bphi X)=\sqrt{\alpha}-3
\end{equation*}
  \item[(ii)] If $|I_{M}|<1$, the sectional curvature of the  paraSasakian metric on $M$, defined by  \eqref{parasasaki}, is:
  \begin{equation}\label{parasasaki-sectional}
  \w{K}(X,Y)=\left\{ \begin{array}{ll}
  -1 & \mbox{if } X=\xi \mbox{ or } Y=\xi,\\
  -\sqrt{-\alpha}   & \mbox{if } X,Y \in\Gamma({\mathcal D}_{h}(\lambda))   \mbox{ or } X,Y \in \Gamma({\mathcal D}_{h}(-\lambda)),\\
  -(\sqrt{-\alpha}-3) \w{g} (X,\wphi Y)^2   & \mbox{if } X \in \Gamma({\mathcal D}_{h}(\pm\lambda)), Y \in \Gamma({\mathcal D}_{h}(\mp\lambda)),
  \end{array} \right.
  \end{equation}
  where $X$ and $Y$ are mutually orthogonal, unit vector fields with respect to $\w{g}$.  In particular, if  $X \in\Gamma({\mathcal D}_{h}(\lambda))$    or  $X \in\Gamma({\mathcal
D}_{h}(-\lambda))$,  the $\w{\phi}$-sectional curvature is given by
\begin{equation*}
\w{K}(X,\wphi X)=-(\sqrt{-\alpha}-3)
\end{equation*}
\end{enumerate}
\end{corollary}
\begin{proof}
(i) By direct computation using equation \eqref{sasaki-R}.\\
(ii) As $\w{g}$ is a  semi-Riemannian metric, we first have to check that   the plane fields considered in \eqref{parasasaki-sectional} are non-degenerate,
i.e. $\w{g}(X,X)\w{g}(Y,Y)-\w{g}(X,Y)^{2}\neq 0$.  Notice that for any $X\in\Gamma({\mathcal D}_{h}(\lambda))$ and $Y\in\Gamma({\mathcal
D}_{h}(-\lambda))$, $\w{g}(X,X)=-\frac{2-\mu+2\lambda}{\sqrt{-\alpha}}g(X,X)<0$ and $\w{g}(Y,Y)=-\frac{2-\mu-2\lambda}{\sqrt{-\alpha}}g(Y,Y)>0$, hence no
vector field which is tangent to the distributions ${\mathcal D}_{h}(\lambda)$ and ${\mathcal D}_{h}(-\lambda)$ is $\w{g}$-isotropic. Furthermore,
${\mathcal D}_{h}(\lambda)$ and ${\mathcal D}_{h}(-\lambda)$ are $\w{g}$-orthogonal. Consequently one finds that all the above $2$-planes $\{X,Y\}$ are
non-degenerate if $X\in\Gamma({\mathcal D}_{h}(\pm\lambda))$ and $Y\in\Gamma({\mathcal D}_{h}(\mp\lambda))$. Whereas, if $X,Y\in\Gamma({\mathcal
D}_{h}(\pm\lambda))$ one has
\begin{gather*}
\w{g}(X,X)\w{g}(Y,Y)-\w{g}(X,Y)^{2}=\left(-\frac{2-\mu\pm2\lambda}{\sqrt{-\alpha}}\right)^{2}\left(g(X,X)g(Y,Y)-g(X,Y)^{2}\right)\\
\w{g}(X,X)\w{g}(\xi,\xi)-\w{g}(X,\xi)^2=-\frac{2-\mu\pm
2\lambda}{\sqrt{-\alpha}}g(X,X)
\end{gather*}
so that also in these cases the $2$-planes $\{X,Y\}$ and
$\{X,\xi \}$  are non-degenerate. Therefore, all the above $2$-planes
are non-degenerate and it makes sense to compute the sectional
curvature for such $2$-planes. The rest follows straightforwardly
from the formula \eqref{parasasaki-R}.
\end{proof}

We now discuss some consequences of Theorems \ref{main1} and \ref{main2}. First notice that \eqref{sasaki-ricci} implies that the metric defined in
\eqref{sasaki} is Sasaki-Einstein if and only if $-\epsilon n\sqrt{\alpha}+2n+2=0$, which is equivalent to requiring that $I_M > 1$ and
\begin{equation*}
(2-\mu+2\lambda)(2-\mu-2\lambda)=4 \left(1+\frac{1}{n} \right)^{2}.
\end{equation*}
Whereas if $I_{M}<-1$ the structure can never be Einstein. Analogously, by \eqref{parasasaki-ricci} one has that the paraSasakian metric defined in
\eqref{parasasaki} is Einstein if and only if $n\sqrt{-\alpha}-2n-3=0$, which is equivalent to
\begin{equation*}
(2-\mu+2\lambda)(2-\mu-2\lambda)=-\left(2+\frac{3}{n} \right)^{2}.
\end{equation*}

In particular we deduce the following corollary.

\begin{corollary}\label{einstein3}
The tangent sphere bundle $T_{1}N$ of any space form $N$ of constant sectional curvature $c \neq 1$ admits a canonical $\eta$-Einstein Sasakian metric
$\overline{g}$ or paraSasakian metric $\w{g}$, according to the circumstance that $c>0$ or $c<0$, respectively. The  Ricci tensors of such metrics are then
given by
\begin{align}
\overline{\emph{Ric}} &= 2(2n\sqrt{c}-1)\overline{g} +
2(-2n\sqrt{c}+n+1)\eta\otimes\eta, \label{ricci1}\\
\w{\emph{Ric}} &= (-4n\sqrt{-c}+3)\w{g} +
(4n\sqrt{-c}-2n-3)\eta\otimes\eta, \label{ricci2}
\end{align}
where $\dim(N)=n+1$.  Moreover, $\overline{g}$ is Einstein if and
only $\dim(N)>2$ and $c=\frac{1}{4}\left(1+\frac{1}{n}\right)^{2}$,
and $\w{g}$ is Einstein if and only if
$c=-\frac{1}{16}\left(2+\frac{3}{n}\right)^{2}$.
\end{corollary}
\begin{proof}
The tangent sphere bundle $T_{1}N$ of a Riemannian manifold $N$ of
constant curvature $c\neq 1$ is a non-Sasakian contact metric
$(\kappa,\mu)$-space with $\kappa=c(2-c)$, $\mu=-2c$ (cf. $\S$
\ref{first}). Since $I_M=\frac{1+c}{|1-c|}$  one has that if $c>0$
($c\neq 1$) then $I_{M}>1$ and if $c<0$ then $-1<I_{M}<1$. Thus,
according to the circumstance that $N$ has positive or negative
sectional curvature, respectively, we can apply (i) or (ii) of
Theorems \ref{main1} and \ref{main2}, and we can conclude that
$T_{1}N$ admits an $\eta$-Einstein Sasakian or paraSasakian metric.
The expressions \eqref{ricci1} and \eqref{ricci2} for the
corresponding Ricci tensors then follow respectively from
\eqref{sasaki-ricci} and \eqref{parasasaki-ricci}, taking into
account that in this case $\epsilon=1$ and $\alpha=16 c$. Finally,
the last statement easily follows, noting that one has to assume
that $\dim(N)>2$ since otherwise $c=1$.
\end{proof}

\begin{remark}
In \cite{blair78} Blair proved that the standard contact structure
on the tangent sphere bundle of a compact Riemannian manifold  of
nonpositive constant curvature cannot be regular. Since regularity
depends only on the underlying Reeb vector field and the topological
structure of the manifold, we conclude that, under the assumption of
compactness of the base manifold, all the $\eta$-Einstein
paraSasakian structures on tangent sphere bundles stated in
Corollary \ref{einstein3} are not regular.
\end{remark}

It is well known that in any Sasakian manifold the Reeb vector field determines a transversely K\"{a}hler foliation, i.e. the Sasakian structure transfers
to a K\"{a}hler structure on the space of leaves. By using the O'Neill equations for a Riemannian foliation (cf. e.g. \cite{tondeur97}) one can prove that
the space of leaves of any $(\kappa,\mu)$-space such that $|I_{M}|>1$ admits a K\"{a}hler-Einstein structure with positive scalar curvature
$2n^{2}\sqrt{\alpha}$ if $I_{M}>1$ and with  negative scalar curvature $-2n^{2}\sqrt{\alpha}$ if $I_{M}<-1$. It can  never be Calabi-Yau because $\alpha$
cannot vanish since $I_M \neq \pm 1$. This is another geometric interpretation of the sign of the Boeckx invariant of a $(\kappa,\mu)$-space.

The geometric behavior of the $(\kappa,\mu)$-space  in the positive
case seems to be very different from the negative one. In fact, when
$I_M>1$ we can apply a theorem of Tanno (\cite{tanno67}) getting
the following result.

\begin{corollary}\label{topology1}
Every contact metric $(\kappa,\mu)$-space $(M,\varphi,\xi,\eta,g)$
with $I_M> 1$ admits a Sasaki-Einstein metric compatible with the
contact form $\eta$.
\end{corollary}
\begin{proof}
By \cite[Proposition 5.3]{tanno67} we know that any $\eta$-Einstein
K-contact metric such that $a$ is constant and is greater than $-2$
can be $\mathcal D$-homothetically transformed in an Einstein
metric. Here $a$ is the function appearing in the $\eta$-Einstein
condition \eqref{etaeinstein1}.
\end{proof}


Similar considerations can be done for the paraSasakian case, that is when $|I_{M}|<1$. In such a case the foliation defined by $\xi$ turns out to be
transversely paraK\"{a}hler, i.e. it is locally described by a family of semi-Riemannian submersions such that the paraSasakian structure on the total
space transfers to a paraK\"{a}hler structure on the base. Then one can see that the paraSasakian metric $\w{g}$ locally projects to a
paraK\"{a}hler-Einstein metric on the leaf space with Einstein constant $1-n\sqrt{-\alpha}$. In particular, the Reeb foliation is transversely
paraK\"{a}hler Ricci-flat if and only if $\sqrt{-\alpha}=\frac{1}{n}$. The last is a ``key value'' according to the following corollary.

\begin{corollary}\label{para-einstein}
Let $(M,\varphi,\xi,\eta,g)$ be a contact metric
$(\kappa,\mu)$-space such that $|I_{M}|<1$. If
\begin{equation*}
(2-\mu+2\lambda)(2-\mu-2\lambda)\neq -\frac{1}{n^2}
\end{equation*}
then $M$ admits a paraSasaki-Einstein structure.
\end{corollary}
\begin{proof}
By \cite[Theorem 4.8]{zamkovoy} we know that every  $\eta$-Einstein paraSasakian manifold of dimension $2n+1$  whose scalar curvature is different from
$2n$, admits a paraSasaki-Einstein structure obtained by a $\mathcal D$-homothetic deformation. Since in our case, by Theorem \ref{main2}, the scalar
curvature of $(M,\w{g})$ is given by $2n(2-n\sqrt{-\alpha})$, the assertion easily follows.
\end{proof}

Corollaries \ref{einstein3}, \ref{topology1} and \ref{para-einstein} give the following consequence
on tangent sphere bundles.

\begin{corollary}
Let $N$ be a space form of constant sectional curvature $c$ and
dimension $n+1$.
\begin{enumerate}
  \item[(i)] If $c>0$, $c\neq 1$ then the tangent sphere bundle $T_{1}N$ carries a Sasaki-Einstein
metric.
  \item[(ii)] If $c<0$, $c\neq -\frac{1}{16n^2}$ then the tangent sphere bundle $T_{1}N$ carries a
paraSasaki-Einstein metric.
\end{enumerate}
\end{corollary}

Let us recall (cf. \cite{baum00}, \cite{boh03}) that a \emph{Lorentzian Sasakian} manifold is a  Lorentzian manifold $(M,g)$ of dimension $2n+1$, endowed
with a time-like vector field $\xi$ with $g(\xi,\xi)=-1$, such that the tensor field $\phi=-\nabla\xi$ satisfies the conditions $\phi^{2}X=-X-g(X,\xi)\xi$
and $(\nabla_{X}\phi)Y=g(Y,\xi)X-g(X,Y)\xi$. If, in addition, the Lorentzian metric $g$ is Einstein, the manifold is said to be \emph{Lorentzian
Sasaki-Einstein}.

Now, for the case $I_M<-1$ we can state the following result, which is a consequence of Theorem \ref{main2} and \cite[Corollary 24]{galicki06}.

\begin{corollary}
Every contact metric $(\kappa,\mu)$-space with $I_{M}<-1$ admits a
Lorentzian Sasaki-Einstein structure.
\end{corollary}

Another consequence in the case $I_M>1$ deals with the notion of
Weyl structure. Recall that (cf. \cite{cp99}) a \emph{Weyl
structure} on a manifold $M$ of dimension $m \geq 3$ is defined by a
pair $W = ([g],D)$, where $[g]$ is a conformal class of Riemannian
metrics and $D$ is the unique torsion-free connection, called
\emph{Weyl connection}, satisfying
\begin{equation}\label{weil2}
Dg = -2\theta\otimes g.
\end{equation}
for some $1$-form $\theta$.
Then $(M, [g],D)$ is said to be \emph{Einstein-Weyl} if there exists a smooth function $\Lambda$ on $M$ such that $\textrm{Ric}^{D}(X, Y) +
\textrm{Ric}^{D}(Y,X) = \Lambda g(X, Y)$, where $\textrm{Ric}^{D}$ denotes the Ricci tensor with respect to the connection $D$. Since the condition
\eqref{weil2} is invariant under Weyl transformations $g'=e^{2f}g$, $\theta'=\theta+df$, with $f\in C^{\infty}(M)$, one sometimes abuses the terminology by
choosing a Riemannian metric in $[g]$ and referring to the pair $W = (g, \theta)$ as a Weyl structure. \ The existence of Einstein-Weyl structures on
almost contact metric manifolds has been recently investigated  by several authors (\cite{gosh09}, \cite{matzeu2011}, \cite{narita97}).

Now as a consequence of Theorem \ref{main2}  and \cite[Corollary
62]{galicki06} we get the following result.

\begin{corollary}\label{sasaki-weil1}
Every non-Sasakian  contact $(\kappa,\mu)$-space of dimension $2n+1\geq 5$ such that $I_{M}>1$  admits an Einstein-Weyl  structure $W=(\bar{g}',\theta)$,
where $\theta=\tau \eta$, $\tau\in\mathbb{R}$, and  $\bar{g}'$ is a Sasakian metric $\mathcal D$-homothetic to that defined by \eqref{sasaki}. Furthermore,
$\bar{g}'$ belongs to the conformal class of $\bar{g}$ provided that
\begin{equation}\label{weil3}
(2-\mu)^{2}-4(1-\kappa)>4\left(1+\frac{1}{n}\right)^{2}
\end{equation}
\end{corollary}
\begin{proof}
Corollary 62 \ of \ \cite{galicki06} \ presents \  a \ necessary \ and \ sufficient \ condition \ for \  a \ K-contact \ manifold \
$(M,\varphi,\xi,\eta,g)$ of dimension greater than $3$ to admit an Einstein-Weyl structure $W=(g,\theta)$, with $\theta=\tau \eta$, $\tau\in\mathbb{R}$.
Such condition is that $M$ is $\eta$-Einstein with Einstein constants $a$, $b$ such that $b<0$, which by \eqref{sasaki-ricci} it is easy to see occurs if
and only if $\epsilon=1$, that is $I_M>1$, and the inequality \eqref{weil3} holds. Now let us apply a ${\mathcal D}_{c}$-homothetic deformation
\eqref{deformation}
where $c$ is any real number chosen so that $c<\frac{\sqrt{\alpha}}{4}$. According to Remark \ref{deformation2} $(\varphi',\xi',\eta',g')$ is still a
$(\kappa',\mu')$-structure, with $\kappa'$ and $\mu'$ given by \eqref{deformation1}, whose associated Sasakian structure
$(\bar{\phi}',\xi',\eta',\bar{g}')$ is in turn
 ${\mathcal D}_{c}$-homothetic to the Sasakian
structure $(\bar{\phi},\xi,\eta,\bar{g})$ associated to $(\varphi,\xi,\eta,g)$. Now, since $\alpha'=\frac{1}{c^{2}}\alpha$ and $c$ was chosen in such a way
that $c<\frac{\sqrt{\alpha}}{4}$ we have that $\sqrt{\alpha'}=\frac{1}{c}\sqrt{\alpha}>4>2\left(1+\frac{1}{n}\right)$. Thus, according to Theorems
\ref{main2}, $\bar{g}'$ is an $\eta'$-Einstein Sasakian metric such that the second Einstein constant is negative. Consequently, we can apply
\cite[Corollary 62]{galicki06} and conclude that $M$ admits  an Einstein-Weyl structure $W=(\bar{g}',\theta)$, where $\theta=\tau' \eta'=c\tau'\eta$,
$\tau'\in\mathbb{R}$. Clearly, if the further assumption \eqref{weil3} is satisfied there is no need to apply a $\mathcal D$-homothetic deformation and the
final assertion of the theorem follows.
\end{proof}

Furthermore, by a similar reasoning to the proof of
Corollary \ref{sasaki-weil1} and by using now \cite[Theorem
63]{galicki06} (or \cite[Theorem 1]{gosh09}) we get the following
result.

\begin{corollary}
Every non-Sasakian  contact $(\kappa,\mu)$-space of dimension $2n+1\geq 5$ such that $I_{M}>1$  admits both the Einstein-Weyl structure
$W^{+}=(\bar{g}',\theta)$ and $W^{-}=(\bar{g}',-\theta)$, for some $1$-form  $\theta$, where $\bar{g}'$ is a Sasakian metric $\mathcal D$-homothetic to the
canonical Sasakian metric $\bar{g}$ defined by \eqref{sasaki}. Moreover, if the inequality \eqref{weil3} holds, then $\bar{g}'$ belongs to the conformal
class of $\bar{g}$.
\end{corollary}

Taking into account that for the tangent sphere bundle of a
Riemannian manifold $N$ of constant sectional curvature $c$ one has
$(2-\mu)^{2}-4(1-\kappa)=16 c$, we can state the following corollary.

\begin{corollary}
Let  $T_{1}N$ be the tangent sphere bundle of a space form $N$ of
constant sectional curvature $c>0$, $c\neq 1$, such that
$\dim(N)=n+1>2$, and let $(\bar{g},\xi,\eta,\bar\phi)$ be the Sasakian
metric on $T_{1}N$ stated by Theorem \ref{sasaki}, where $\xi$ is
twice the geodesic flow. Then
\begin{enumerate}
  \item[(i)] $T_{1}N$ carries an Einstein-Weyl  structure
$W=(\bar{g}',\theta)$, where $\theta=\tau \eta$ for some
$\tau\in\mathbb{R}$ and  $\bar{g}'$ is a Sasakian metric $\mathcal
D$-homothetic to $\bar{g}$.
  \item[(ii)] $T_{1}N$ admits both the Einstein-Weyl  structure
$W^{+}=(\bar{g}'',\theta)$ and $W^{-}=(\bar{g}'',-\theta)$, for some
$1$-form  $\theta$ not necessarily proportional to $\eta$, where
$\bar{g}''$ is a Sasakian metric $\mathcal D$-homothetic to $\bar{g}$.
\end{enumerate}
Furthermore, if the inequality
\begin{equation*}
c>\frac{1}{4}\left(1+\frac{1}{n}\right)^{2}
\end{equation*}
holds, then the metrics $\bar{g}'$ and $\bar{g}''$ stated in (i) and
(ii), respectively, belong to the conformal class of $\bar{g}$.
\end{corollary}

We close the section with a remark concerning the case $|I_M|<1$. By
\cite[Theorem 5.6]{mino1} every contact metric
$(\kappa,\mu)$-space $(M,\varphi,\xi,\eta,g)$ such that
$|I_{M}|<1$ carries a canonical sequence
$(\varphi_n)_{n\in\mathbb{N}}$ of contact or paracontact metric
structures defined as follows:
\begin{equation}\label{sequence}
\varphi_{0}:=\varphi, \ \ \
\varphi_{1}:=\frac{1}{2\sqrt{1-\kappa}}{\mathcal
L}_{\xi}\varphi_{0}, \ \ \
\varphi_{n}:=\frac{1}{\sqrt{4(1-\kappa)-(2-\mu)^2}}{\mathcal
L}_{\xi}\varphi_{n-1}, \ n\geq 2.
\end{equation}
By defining
\begin{equation*}
g_{n}:=\left\{
         \begin{array}{ll}
           -d\eta(\cdot,\varphi_{n}\cdot)  + \eta\otimes\eta, & \hbox{ if $n$ is even} \\
           d\eta(\cdot,\varphi_{n}\cdot)  + \eta\otimes\eta, & \hbox{ if $n$ is odd}
         \end{array}
       \right.
\end{equation*}
one can prove that, for each $n\geq 1$,
$(\varphi_{n},\xi,\eta,g_{n})$ is a contact metric
$(\kappa_{n},\mu_{n})$-structure if $n$ is even and a paracontact
metric $(\kappa_{n},\mu_{n})$-structure if $n$ is odd, where
\begin{equation*}
\kappa_{n}=\left\{
             \begin{array}{ll}
               \kappa+\left(1-\frac{\mu}{2}\right)^2, & \hbox{if $n$ is even} \\
               \kappa-2+\left(1-\frac{\mu}{2}\right)^2, & \hbox{if $n$ is odd}
             \end{array}
           \right.
\end{equation*}
and $\mu_{n}=2$. In fact, by using \eqref{lieh}, one can see that
$\varphi_{2n}=\varphi_{2}$ and $\varphi_{2n+1}=\varphi_{1}$ for any
$n\geq 1$.

Since by assumption the Boeckx invariant of the starting contact
metric $(\kappa,\mu)$-structure $(\varphi,\xi,\eta,g)$ satisfies
$|I_M|<1$, we can apply Theorems \ref{main1}--\ref{main2}, so that
$M$ carries a canonical  $\eta$-Einstein paraSasakian metric
$(\w{\phi},\xi,\eta,\w{g})$. On the other hand the same condition is
satisfied also by  the Boeckx invariant of the contact metric
$(\kappa_{2n},\mu_{2n})$-structures for any $n\geq 1$, since
$\mu_{n}=2$. Thus, applying again Theorems \ref{main1}--\ref{main2},
 $M$ carries an  $\eta$-Einstein paraSasakian
structure $(\w{\phi}_{2n},\xi,\eta,\w{g}_{2n})$ for each $n\geq 1$.
Actually we will prove that
$(\w{\phi}_{2n},\xi,\eta,\w{g}_{2n})=(\w{\phi},\xi,\eta,\w{g})$,
 that is the paraSasakian structure $(\w{\phi},\xi,\eta,\w{g})$ stated in
Theorem \ref{main1} is invariant under the canonical sequence
\eqref{sequence}. Indeed it is enough to check this for $n=2$. We
have that
\begin{align*}
\varphi_{2}=\frac{1}{\sqrt{4(1-\kappa)-(2-\mu)^{2}}}{\mathcal
L}_{\xi}\varphi_{1}=\frac{\left((2-\mu)\varphi
h + 2(1-\kappa)\varphi\right)}{\sqrt{(1-\kappa)(4(1-\kappa)-(2-\mu)^{2})}}.
\end{align*}
Then, by \eqref{parasasaki} and taking into account the fact that
$\kappa_{2}=\kappa+\left(1-\frac{\mu}{2}\right)^{2}$, $\mu_{2}=2$, one has
\begin{equation}\label{sequence1}
\w{\phi}_{2}=\frac{1}{\sqrt{4(1-\kappa_{2})-(2-\mu_{2})^2}}\left((2-\mu_{2})\varphi_{2}+2\varphi_{2}h_{2}\right)
=\frac{2}{\sqrt{4(1-\kappa)-(2-\mu)^2}}\varphi_{2}h_{2}.
\end{equation}
Since $h_{2}=\sqrt{1-I_{M}^{2}}h$ (\cite[p. 275]{mino1}), formula
\eqref{sequence1} becomes
\begin{equation*}
\w{\phi}_{2}=\frac{2}{\sqrt{-\alpha}}\left(-\frac{\lambda(2-\mu)\sqrt{1-I_{M}^{2}}} {\sqrt{-\alpha}}\varphi^{3}+
\frac{2\lambda\sqrt{1-I_{M}^{2}}}{\sqrt{-\alpha}}\varphi h\right)=\frac{1}{\sqrt{-\alpha}}\left((2-\mu)\varphi+2\varphi h\right)=\w{\phi}.
\end{equation*}
Therefore $\w{\phi}_2=\w{\phi}$ and, as the contact form is the same in both structures, we also have that $\bar{g}_2=\bar{g}$.

\section{$(\kappa,\mu)$-structures on K-contact and K-paracontact manifolds}

In this section we study the converse, in some sense,  of Theorem
\ref{main1}. Namely we find sufficient conditions for a Sasakian or
 paraSasakian manifold to admit a $(\kappa,\mu)$-structure compatible with the same underlying contact
form. This will also allow us to give a geometrical
interpretation of the Sasakian and paraSasakian metrics defined by
\eqref{formulametrica1} and \eqref{formulametrica2}, respectively.
Actually we will show that the assumption of normality is too
strong, so that one can at first assume the manifold to be K-contact
or K-paracontact.

Recall that  a \emph{Legendre foliation} (cf. \cite{blairbook}) on a contact manifold $(M,\eta)$ is nothing but an integrable $n$-dimensional subbundle of
the contact distribution.  Legendre foliations have been extensively investigated in recent years from various points of views. In particular  Pang
(\cite{pang}) provided a classification of Legendre foliations by means of a bilinear symmetric form $\Pi_{\mathcal F}$ on the tangent bundle of the
foliation ${\mathcal F}$, defined by
\begin{equation*}
\Pi_{\mathcal F}\left(X,X'\right)=-\left({\mathcal L}_{X}{\mathcal L}_{X'}\eta\right)\left(\xi\right)=2d\eta([\xi,X],X').
\end{equation*}
He called a Legendre foliation \emph{non-degenerate}, \emph{degenerate} or \emph{flat} according to the circumstance that the bilinear form $\Pi_{\mathcal
F}$ is non-degenerate, degenerate or vanishes identically, respectively. For  a  non-degenerate Legendre  foliation  $\mathcal F$, Libermann
(\cite{libermann}) defined also  a linear  map $\Lambda_{\mathcal F}:TM\longrightarrow T{\mathcal F}$, whose kernel is ${T\mathcal F}\oplus\mathbb{R}\xi$,
such that
\begin{equation}\label{lambda}
\Pi_{\mathcal F}(\Lambda_{\mathcal F} Z,X)=d\eta(Z,X)
\end{equation}
for any $Z\in\Gamma(TM)$, $X\in\Gamma(T{\mathcal F})$. The operator $\Lambda_{\mathcal F}$ is surjective and  satisfies $(\Lambda_{\mathcal F})^2=0$ and
\begin{equation}\label{proplambda}
\Lambda_{\mathcal F}[\xi,X]=\frac{1}{2}X
\end{equation}
for all $X\in\Gamma(T{\mathcal F})$.  Then  one can extend $\Pi_{\mathcal F}$ to a symmetric bilinear form defined on all $TM$ by setting
\begin{equation*}
\overline\Pi_{\mathcal F}(Z,Z'):=\left\{
                                   \begin{array}{ll}
                                     \Pi_{\mathcal F}(Z,Z') & \hbox{if $Z,Z'\in\Gamma(T{\mathcal F})$} \\
                                     \Pi_{\mathcal F}(\Lambda_{\mathcal F} Z,\Lambda_{\mathcal F} Z') & \hbox{otherwise.}
                                   \end{array}
                                 \right.
\end{equation*}
Now let $(M,\varphi,\xi,\eta,g)$ be a non-Sasakian contact metric $(\kappa,\mu)$-space. Being $n$-dimensional integrable subbundles of the contact
distribution, the eigenspaces of $h$, ${\mathcal D}_{h}(\lambda)$ and ${\mathcal D}_{h}(-\lambda)$, define two mutually orthogonal Legendre foliations on
$M$. Such a foliated structure of a $(\kappa,\mu)$-space was studied in \cite{mino3}. In particular,  more explicit formulas for the Pang invariants  of
${\mathcal D}_{h}(\lambda)$ and ${\mathcal D}_{h}(-\lambda)$ were found, namely
\begin{equation}\label{panginvariant1}
\Pi_{{\mathcal D}_{h}(\lambda)}=(2-\mu+2\lambda)g|_{{\mathcal D}_{h}(\lambda)\times{\mathcal D}_{h}(\lambda)}, \ \ \  \Pi_{{\mathcal D}_{h}(-\lambda)}=(2-\mu-2\lambda)g|_{{\mathcal D}_{h}(-\lambda)\times{\mathcal D}_{h}(-\lambda)},
\end{equation}
from which it follows that ${\mathcal D}_{h}(\lambda)$ and ${\mathcal D}_{h}(-\lambda)$ are both non-degenerate if and only if $I_{M}\neq\pm 1$.

Let us assume that $|I_{M}|> 1$. By using \eqref{panginvariant1} and
\eqref{formulametrica1} one has, for all $X,X'\in\Gamma({\mathcal
D}_{h}(\lambda))$,
\begin{align*}
\Pi_{{\mathcal D}_{h}(\lambda)}(X,X')&=(2-\mu)g(X,X')+2g(\lambda X,X')=(2-\mu)g(X,X')+2g(hX,X')=\epsilon \sqrt{\alpha}\bar{g}(X,X'),
\end{align*}
where $\alpha=(2-\mu)^{2}-4(1-\kappa)$ and $\epsilon$ is the sign of
$I_{M}$, according to the notation of $\S$ \ref{second}.
Analogously, for all $Y,Y'\in\Gamma({\mathcal D}_{h}(-\lambda))$:
\begin{align*}
\Pi_{{\mathcal D}_{h}(-\lambda)}(Y,Y')=\epsilon \sqrt{\alpha}\bar{g}(Y,Y')
\end{align*}
Whereas, if $|I_{M}|<1$ one has
\begin{equation*}
\Pi_{{\mathcal D}_{h}(\lambda)}(X,X')=-\sqrt{-\alpha}\w{g}(X,X'), \ \ \ \Pi_{{\mathcal D}_{h}(-\lambda)}(Y,Y')=-\sqrt{-\alpha}\w{g}(Y,Y').
\end{equation*}
Therefore, geometrically the Sasakian metric $\bar{g}$ in the case $|I_{M}|>1$ and the paraSasakian metric $\w{g}$ in the case $|I_{M}|<1$ represent, up to
a constant factor, the Pang invariant of the foliations ${\mathcal D}_{h}(\lambda)$ and ${\mathcal D}_{h}(-\lambda)$ when restricted to the leaves.\
Moreover, notice that since ${\mathcal D}_{h}(\lambda)$ and ${\mathcal D}_{h}(-\lambda)$ are totally geodesic foliations (\cite[Proposition 3.7]{blair95})
and because of \eqref{sasaki-nabla} and \eqref{parasasaki-nabla}, they are still totally geodesic with respect to  $\bar{g}$ and $\w{g}$. Now we prove that
the above properties determine uniquely the contact metric $(\kappa,\mu)$-structure.

\begin{theorem}\label{reves1}
Let $(M,\bar\phi,\xi,\eta,\bar{g})$ be a K-contact manifold endowed
with two totally geodesic, mutually orthogonal Legendre foliations
${\mathcal F}_{1}$ and ${\mathcal F}_{2}$. Assume that
\begin{equation}\label{panginvariant2}
\Pi_{{\mathcal F}_{1}}=\epsilon\sqrt{ab}\bar{g}|_{{\mathcal
F}_{1}\times{\mathcal F}_{1}}, \ \ \ \Pi_{{\mathcal
F}_{2}}=\epsilon\sqrt{ab}\bar{g}|_{{\mathcal F}_{2}\times{\mathcal
F}_{2}}
\end{equation}
for some real numbers $a$ and $b$ such that $a\neq b$ and $a\cdot
b>0$, where
\begin{equation*}
\epsilon=\left\{
           \begin{array}{ll}
             1, & \hbox{if $a>0$, $b>0$} \\
             -1, & \hbox{if $a<0$, $b<0$.}
           \end{array}
         \right.
\end{equation*}
Then $M$ admits a contact metric
$(\kappa_{a,b},\mu_{a,b})$-structure
$(\varphi_{a,b},\xi,\eta,g_{a,b})$, compatible with the original
contact form $\eta$, where
\begin{equation}\label{valores}
\kappa_{a,b}=1-\frac{(a-b)^{2}}{16}, \ \ \
\mu_{a,b}=2-\frac{a+b}{2}.
\end{equation}
Furthermore, $(M,\bar\phi,\xi,\eta,\bar{g})$  is Sasakian and
$\eta$-Einstein with Ricci tensor
\begin{equation}\label{formula-ricci1}
\overline{\emph{Ric}}=(\epsilon n \sqrt{ab} -2)\overline{g} + (-\epsilon n\sqrt{ab} + 2n + 2)\eta\otimes\eta.
\end{equation}
\end{theorem}
\begin{proof}
First notice that, since ${\mathcal F}_{1}$ and ${\mathcal F}_{2}$
are mutually $\bar{g}$-orthogonal,  $\bar\phi {T\mathcal F}_{1}\subset
{T\mathcal F}_{2}$ and $\bar\phi {T\mathcal F}_{2}\subset {T\mathcal
F}_{1}$. Indeed let us fix
 $X\in\Gamma(T{\mathcal F}_{1})$. Then for any $X'\in\Gamma(T{\mathcal F}_{1})$ one has $\bar{g}(\bar{\phi}X,X')=-d\eta(X,X')=\frac{1}{2}\eta([X,X'])=0$ since ${\mathcal F}_{1}$ is a Legendre foliation. Moreover, obviously $\bar{g}(\bar{\phi}X,\xi)=0$.
Consequently, since by assumption $(T{\mathcal
F}_{1}\oplus\mathbb{R}\xi)^{\perp}=T{\mathcal F}_{2}$, we conclude
that $\bar{\phi}X\in\Gamma(T{\mathcal F}_{2})$. In a similar way one
can prove the other relation. \ Now, taking into account the
decomposition of the tangent bundle of $M$ as $TM=T{\mathcal
F}_{1}\oplus T{\mathcal F}_{2}\oplus\mathbb{R}\xi$, let us define a
$(1,1)$-tensor field $\varphi_{a,b}$ and a Riemannian metric
$g_{a,b}$  on $M$ by
\begin{equation}\label{definicion}
\varphi_{a,b}:=\left\{
           \begin{array}{ll}
             \sqrt{\frac{b}{a}}\bar\phi, & \hbox{on $T{\mathcal F}_{1}$} \\
             \sqrt{\frac{a}{b}}\bar\phi, & \hbox{on $T{\mathcal F}_{2}$} \\
             0, & \hbox{on $\mathbb{R}\xi$}
           \end{array}
         \right. \ \ \
g_{a,b}:=\left\{
     \begin{array}{ll}
       \sqrt{\frac{b}{a}}\bar{g}, & \hbox{on $T{\mathcal F}_{1}\times T{\mathcal F}_{1}$} \\
       \sqrt{\frac{a}{b}}\bar{g}, & \hbox{on $T{\mathcal F}_{2}\times T{\mathcal F}_{2}$} \\
       \eta\otimes\eta, & \hbox{otherwise}
     \end{array}
   \right.
\end{equation}
and \ extend it \ by \ linearity. \ A \ straightforward \
computation \ shows \ that \ $\varphi_{a,b}^{2}=-I+\eta\otimes\xi$ \
and \
$g_{a,b}(\varphi_{a,b}\cdot,\varphi_{a,b}\cdot)=g_{a,b}-\eta\otimes\eta$.
Let us check that $(\varphi_{a,b},\xi,\eta,g_{a,b})$ is a contact
metric structure, i.e.
\begin{equation}\label{contactometrica}
d\eta(X,Y)=g_{a,b}(X,\varphi_{a,b} Y)
\end{equation}
for all $X,Y\in\Gamma(TM)$. If either $X$ or $Y$ belongs to
$\mathbb{R}\xi$, \eqref{contactometrica} is trivially satisfied,
both  sides being zero. Let $X,Y\in\Gamma(T{\mathcal F}_1)$. Then
$d\eta(X,Y)=-\frac{1}{2}\eta([X,Y])=0$ because ${\mathcal F}_{1}$ is   a
Legendre foliation. On the other hand, since
$\bar{\phi}Y\in\Gamma(T{\mathcal F}_{2})$, then
 $g_{a,b}(X,\varphi_{a,b} Y)=\sqrt{\frac{b}{a}}g_{a,b}(X,\bar{\phi}Y)=\sqrt{\frac{b}{a}}\eta(X)\eta(\bar\phi Y)=0=d\eta(X,Y)$. Analogously, if $X,Y\in\Gamma(T{\mathcal F}_{2})$, then $g_{a,b}(X,\bar\phi Y)=0=d\eta(X,Y)$. It remains the case
when $X\in\Gamma(T{\mathcal F}_{1})$ and $Y\in\Gamma(T{\mathcal F}_{2})$. In that case by using \eqref{definicion} one has $g_{a,b}(X,\varphi_{a,b}Y)=
\sqrt{\frac{a}{b}} \sqrt{\frac{b}{a}} \bar{g}(X,\bar\phi Y)=\bar{g}(X,\bar\phi Y)=d\eta(X,Y)$, since $(\bar{\phi},\xi,\eta,\bar{g})$ is a contact metric
structure. \ Notice that, directly by the definition of $g_{a,b}$, ${\mathcal F}_{1}$ and  ${\mathcal F}_{2}$ are  $g_{a,b}$-orthogonal, so that  the
tensor field $\varphi_{a,b}$ also maps $T{\mathcal F}_{1}$ on $T{\mathcal F}_{2}$ and $T{\mathcal F}_{2}$ on $T{\mathcal F}_{1}$. \ Finally we prove that
$(M,\varphi_{a,b},\xi,\eta,g_{a,b})$ is a $(\kappa,\mu)$-space. First of all, we compute the operator $h_{a,b}$ of the contact metric structure
$(\varphi_{a,b},\xi,\eta,g_{a,b})$. We have, for any $X\in\Gamma(T{\mathcal F}_1)$,
\begin{align}\label{acce1}
2h_{a,b}X&=[\xi,\varphi_{a,b} X]-\varphi_{a,b}([\xi,X]_{{\mathcal F}_1})-\varphi_{a,b}([\xi,X]_{{\mathcal F}_2}) \nonumber\\
&=\sqrt{\frac{b}{a}}[\xi,\bar\phi X] - \sqrt{\frac{b}{a}}\bar\phi([\xi,X]_{{\mathcal F}_{1}}) - \sqrt{\frac{a}{b}}\bar\phi([\xi,X]_{{\mathcal F}_{2}}) \nonumber\\
&=\sqrt{\frac{b}{a}}[\xi,\bar{\phi}X]_{{\mathcal F}_{1}} + 2\sqrt{\frac{b}{a}}(\bar{h}X)_{{\mathcal F}_{2}} - \sqrt{\frac{a}{b}}\bar\phi([\xi,X]_{{\mathcal
F}_{2}}) \nonumber\\
& =\left(\sqrt{\frac{b}{a}}-\sqrt{\frac{a}{b}}\right)(\bar\phi[\xi,X])_{{\mathcal F}_{1}},
\end{align}
where we have used the K-contact condition $\bar{h}=0$ and we have
decomposed the vector field $[\xi,X]$ according to the decomposition
$TM=T{\mathcal F}_{1}\oplus T{\mathcal F}_{2}\oplus \mathbb{R}\xi$.
In particular, from \eqref{acce1}
 it follows that $h_{a,b}$ maps $T{\mathcal F}_1$ on $T{\mathcal F}_1$. Now we use the assumption \eqref{panginvariant2} for finding a more explicit expression for the Liberman map
$\Lambda_{{\mathcal F}_{1}}:TM\longrightarrow T{\mathcal F}_{1}$. By using \eqref{lambda} we have, for any $X\in\Gamma(T{\mathcal F}_{1})$ and
$Z\in\Gamma(TM)$, that $\epsilon\sqrt{ab}\bar{g}(\Lambda_{{\mathcal F}_{1}}Z,X)=\Pi_{{\mathcal F}_{1}}(\Lambda_{{\mathcal
F}_{1}}Z,X)=d\eta(Z,X)=-\bar{g}(X,\bar\phi Z)$, from which it follows that
\begin{equation}\label{lambda1}
\Lambda_{{\mathcal F}_{1}}Z=-\epsilon\frac{1}{\sqrt{ab}}(\bar\phi
Z)_{{\mathcal F}_{1}}.
\end{equation}
Then, by using \eqref{proplambda} and \eqref{lambda1}, we get
\begin{equation}\label{lambda2}
(\bar\phi[\xi,X])_{{\mathcal
F}_{1}}=-\epsilon\sqrt{ab}\Lambda_{{\mathcal
F}_{1}}[\xi,X]=-\epsilon\frac{\sqrt{ab}}{2}X.
\end{equation}
Therefore, due to \eqref{lambda2}, formula \eqref{acce1} becomes
\begin{equation}\label{acce2}
h_{a,b}X=-\epsilon\frac{\sqrt{ab}}{4}\left(\sqrt{\frac{b}{a}}-\sqrt{\frac{a}{b}}\right)X=\epsilon\frac{|a|-|b|}{4}X=\frac{a-b}{4}X,
\end{equation}
since $\epsilon|a|=a$ and $\epsilon|b|=b$. Due to
$h_{a,b}\varphi_{a,b}=-\varphi_{a,b} h_{a,b}$, from \eqref{acce2} it
follows that
\begin{equation}\label{acce3}
h_{a,b}Y=-\frac{a-b}{4}Y
\end{equation}
for any $Y\in\Gamma(T{\mathcal F}_2)$. Thus we have proved that the tangent bundles of the foliations ${\mathcal F}_{1}$ and ${\mathcal F}_{2}$ coincide
with the eigendistributions of $h_{a,b}$ corresponding to the constant eigenvalues $\pm\lambda_{a,b}$, where $\lambda_{a,b}:=\frac{|a-b|}{4}$. \ Let
$\nabla^{bl}$ be the bi-Legendrian connection associated to $({\mathcal F}_{1},{\mathcal F}_{2})$ (cf. \cite{mino5}). By definition $\nabla^{bl}$ is the
unique linear connection on $M$ such that the following conditions hold:
\begin{enumerate}
  \item[(i)] $\nabla^{bl}{\mathcal F}_{1}\subset {\mathcal F}_{1}$,  $\nabla^{bl}{\mathcal F}_{2}\subset {\mathcal F}_{2}$
  \item[(ii)] $\nabla^{bl}\eta=0$, $\nabla^{bl}d\eta=0$,
  \item[(iii)] $T^{bl}(X,Y)=2d\eta(X,Y)\xi$ \ for any $X\in\Gamma(T{\mathcal F}_{1})$, $Y\in\Gamma(T{\mathcal F}_{2})$,  \\ $T^{bl}(Z,\xi)=[\xi,Z_{{\mathcal F}_{1}}]_{{\mathcal F}_{2}}+[\xi,Z_{{\mathcal F}_{2}}]_{{\mathcal F}_{1}}$ \ for any $Z\in\Gamma(TM)$,
\end{enumerate}
where $T^{bl}$ denotes the torsion tensor field of $\nabla^{bl}$. Actually, because of the integrability of ${T\mathcal F}_{1}$ and ${T\mathcal F}_{2}$,
one has immediately that $T^{bl}(X,X')=0=2d\eta(X,X')\xi$ for any $X,X'\in\Gamma(T{\mathcal F}_1)$ and $T^{bl}(Y,Y')=0=2d\eta(Y,Y')\xi$ for any
$Y,Y'\in\Gamma(T{\mathcal F}_2)$. Thus the first identity in (iii) holds for any two vector fields on the contact distribution. Moreover, since ${\mathcal
F}_1$ and ${\mathcal F}_2$  are assumed to be totally geodesic with respect to $\bar{g}$, due to \cite[Proposition 2.9]{mino5}, we get that
$\nabla^{bl}\bar\phi=0$. Then, because of (i) and \eqref{definicion}, one has that $\nabla^{bl}$ preserves also $\varphi_{a,b}$. Next, as $d\eta =
g_{a,b}(\cdot,\varphi_{a,b}\cdot)$, after a straightforward computation one obtains
\begin{equation*}
\begin{aligned}
(\nabla^{bl}_{Z}g_{a,b})(Z',Z'')=&-(\nabla^{bl}_{Z}d\eta)(Z',\varphi_{a,b}
Z'')-d\eta(Z',(\nabla^{bl}_{Z}\varphi_{a,b})Z'')+\eta(Z'')(\nabla^{bl}_{Z}\eta)(Z')+\eta(Z')(\nabla^{bl}_{Z}\eta)(Z'')
\end{aligned}
\end{equation*}
for all $Z,Z',Z''\in\Gamma(TM)$. Since $\nabla^{bl}$ preserves $\eta$, $d\eta$ and $\varphi_{a,b}$, the above equation implies that $\nabla^{bl}g_{a,b}=0$.
Finally (i) and \eqref{acce2}--\eqref{acce3} easily ensure that $\nabla^{bl}$ preserves also the tensor field $h_{a,b}$. Therefore the bi-Legendrian
connection $\nabla^{bl}$ defined by the foliations ${\mathcal F}_{1}$, ${\mathcal F}_{2}$ satisfies all the conditions required by \cite[Theorem
4.4]{mino6} and we can conclude that $(\varphi_{a,b},\xi,\eta,g_{a,b})$ is a $(\kappa,\mu)$-structure. In order to find explicitly the constants $\kappa$
and $\mu$ (which of course will depend on $a$ and $b$) we notice that, since $\Pi_{{\mathcal F}_{1}}$ is an invariant of the Legendre foliation ${\mathcal
F}_{1}$, the two expressions for $\Pi_{{\mathcal F}_{1}}$ in \eqref{panginvariant1}
 and \eqref{panginvariant2} have to coincide. We thus have, for any $0\neq X\in\Gamma(T{\mathcal F}_{1})$,
\begin{equation*}
(2-\mu_{a,b}+2\lambda_{a,b})g_{a,b}(X,X)=\epsilon\sqrt{ab}\bar{g}(X,X)=\epsilon|a|g_{a,b}(X,X)=a
g_{a,b}(X,X).
\end{equation*}
Thus
\begin{equation}\label{ecuacione1}
2-\mu_{a,b}+2\lambda_{a,b}=a.
\end{equation}
Arguing in a similar way for ${\mathcal F}_2$ one gets
\begin{equation}\label{ecuacione2}
2-\mu_{a,b}-2\lambda_{a,b}=b.
\end{equation}
Then by using \eqref{ecuacione1}--\eqref{ecuacione2} we get the
values \eqref{valores}  for $\kappa_{a,b}$ and $\mu_{a,b}$. \ We will now
prove the last part of the statement. From
\eqref{definicion} and \eqref{ecuacione1}--\eqref{ecuacione2} one
has immediately that for any $X,X'\in\Gamma(T{\mathcal F}_{1})$
\begin{align*}
\bar{g}(X,X')&=\sqrt{\frac{a}{b}}g_{a,b}(X,X')\\
&=\frac{a}{\sqrt{ab}}g_{a,b}(X,X')\\
&=\frac{1}{\sqrt{(2-\mu_{a,b})^2-4(1-\kappa_{a,b})}}\left((2-\mu_{a,b})g_{a,b}(X,X')+2g_{a,b}(h_{a,b}X,X')\right).
\end{align*}
Similarly, one has that
\begin{equation*}
\bar{g}(Y,Y')=\frac{1}{\sqrt{(2-\mu_{a,b})^2-4(1-\kappa_{a,b})}}\left((2-\mu_{a,b})g_{a,b}(Y,Y')+2g_{a,b}(h_{a,b}Y,Y')\right)
\end{equation*}
for any $Y,Y'\in\Gamma(T{\mathcal F}_{2})$. Therefore, as
${T\mathcal F}_{1}={\mathcal D}_{h_{a,b}}(\lambda_{a,b})$ and
${T\mathcal F}_{2}={\mathcal D}_{h_{a,b}}(-\lambda_{a,b})$, we see
that the metric $\bar{g}$ coincides with the Sasakian metric defined
by \eqref{formulametrica1}. Then $\bar{g}$ is also $\eta$-Einstein
with Ricci tensor \eqref{formula-ricci1} because of Theorem
\ref{main2}.
\end{proof}

Notice that each contact metric $(\kappa_{a,b},\mu_{a,b})$-structure $(\varphi_{a,b},\xi,\eta,g_{a,b})$ of the family stated in Theorem \ref{reves1} has a
Boeckx invariant whose absolute value is strictly greater than $1$. Indeed from \eqref{valores} it follows that $I_{M}=\frac{a+b}{|a-b|}$.

Finally, we will examine the paracontact case.

\begin{theorem}
Let $(M,\w\phi,\xi,\eta,\w{g})$ be a K-paracontact manifold endowed
with two totally geodesic, mutually orthogonal Legendre foliations
${\mathcal F}_{1}$ and ${\mathcal F}_{2}$, such that
$\w{g}|_{{\mathcal F}_{1}\times{\mathcal F}_{1}}$ is positive
definite and $\w{g}|_{{\mathcal F}_{2}\times{\mathcal F}_{2}}$ is
negative definite. If
\begin{equation*}
\Pi_{{\mathcal F}_{1}}=-\sqrt{-ab}\w{g}|_{{\mathcal
F}_{1}\times{\mathcal F}_{1}}, \ \ \ \Pi_{{\mathcal
F}_{2}}=-\sqrt{-ab}\w{g}|_{{\mathcal F}_{2}\times{\mathcal F}_{2}}
\end{equation*}
for some real numbers $a$ and $b$ such that $a\neq b$ and $a\cdot
b<0$, then $M$ admits a contact metric $(\kappa,\mu)$-structure
$(\varphi_{a,b},\xi,\eta,g_{a,b})$, compatible with the original
contact form $\eta$, where
\begin{equation*}
\kappa=1-\frac{(a-b)^{2}}{16}, \ \ \ \mu=2-\frac{a+b}{2}.
\end{equation*}
Furthermore, $(M,\bar\phi,\xi,\eta,\bar{g})$  is paraSasakian and
$\eta$-Einstein with Ricci tensor
 \begin{equation*}
    \w{\emph{Ric}}=(-n\sqrt{-ab}+3)\w{g}+(n\sqrt{-ab}-2n-3)\eta\otimes\eta.
    \end{equation*}
\end{theorem}
\begin{proof}
The proof is very similar to that of Theorem \ref{reves1},  the only
difference is the definition of the structure tensors
$\varphi_{a,b}$ and
 $g_{a,b}$, which now are given by
\begin{equation}\label{definicion1}
\varphi_{a,b}:=\left\{
           \begin{array}{ll}
             \frac{\sqrt{-ab}}{a}\w\phi, & \hbox{on $T{\mathcal F}_{1}$} \\
             \frac{\sqrt{-ab}}{b}\w\phi, & \hbox{on $T{\mathcal F}_{2}$} \\
             0, & \hbox{on $\mathbb{R}\xi$}
           \end{array}
         \right. \ \ \
g_{a,b}:=\left\{
     \begin{array}{ll}
       -\frac{\sqrt{-ab}}{a}\w{g}, & \hbox{on $T{\mathcal F}_{1}\times T{\mathcal F}_{1}$} \\
       -\frac{\sqrt{-ab}}{b}\w{g}, & \hbox{on $T{\mathcal F}_{2}\times T{\mathcal F}_{2}$} \\
       \eta\otimes\eta, & \hbox{otherwise}
     \end{array}
   \right.
\end{equation}
The extra-assumption that $\w{g}|_{{\mathcal F}_{1}\times{\mathcal F}_{1}}$ is positive definite and $\w{g}|_{{\mathcal F}_{2}\times{\mathcal F}_{2}}$ is
negative definite ensure that the symmetric tensor $g_{a,b}$ defined by \eqref{definicion1} is a Riemannian metric. The rest of the proof goes as in
Theorem \ref{reves1}, once one notices that \cite[Proposition 2.9]{mino5}, which is used for proving that $(\varphi,\xi,\eta,g)$ is a
$(\kappa,\mu)$-structure, straightforwardly holds also in the context of paracontact metric geometry.
\end{proof}

\end{document}